\documentclass[reqno]{amsart}
\usepackage[utf8]{inputenc}
\usepackage{enumerate}
\usepackage{amsmath, amssymb,amsthm}
\usepackage{mathtools}
\usepackage{leftidx}
\numberwithin{equation}{section}
\usepackage[left=3.75cm, right=3.75cm]{geometry}
\usepackage{hyperref}
\usepackage{xcolor}
\definecolor{my-black}{rgb}{0,0,0}
\definecolor{my-blue}{rgb}{0,0,0.8}
\definecolor{my-red}{rgb}{0.8,0,0} 
\hypersetup{colorlinks,
    linkcolor	= {my-blue},
    citecolor	= {my-blue},
    urlcolor		= {my-blue}
}
\usepackage{graphicx}
\usepackage{tikz-cd}

\theoremstyle{plain} 

\newtheorem{lemma}{Lemma}[section]
\newtheorem{theorem}{Theorem}
\newtheorem{corollary}{Corollary}[section]
\newtheorem{proposition}{Proposition}[section]

\theoremstyle{definition} 

\newtheorem{definition}{Definition}[section]

\theoremstyle{remark}
\newtheorem*{remark-non}{Remark}

\DeclareMathOperator{\real}{Re}
\DeclareMathOperator{\imag}{Im}

\DeclareMathOperator{\SL}{SL}

\DeclareMathOperator{\PSL}{PSL}

\newcommand{\R}{\mathbb{R}}

\newcommand{\C}{\mathbb{C}}
\newcommand{\N}{\mathbb{N}}
\newcommand{\Q}{\mathbb{Q}}
\newcommand{\Z}{\mathbb{Z}}

\renewcommand{\H}{\mathbb{H}}

\title{Fourier interpolation from spheres}
\author{Martin Stoller}
\address{Ecole Polytechnique F{\'e}d{\'e}rale de Lausanne, Lausanne, Switzerland}
\email{martin.stoller@epfl.ch}
\date{}

\begin{document}
\begin{abstract}
In every dimension $d \geq 2$, we give an explicit formula that expresses the values of any Schwartz function on $\R^d$ only in terms of its restrictions, and the restrictions of its Fourier transform, to all origin-centered spheres whose radius is the square root of an integer. We thus generalize an interpolation theorem by Radchenko and Viazovska  \cite{R-V} to higher dimensions. We develop a general tool to translate Fourier uniqueness- and interpolation results for radial functions in higher dimensions, to corresponding results for non-radial functions in a fixed dimension. In dimensions greater or equal to $5$, we solve the radial problem using a construction closely related to classical Poincar{\'e} series. In the remaining small dimensions, we combine this technique with a direct generalization of the Radchenko--Viazovska formula to higher-dimensional radial functions, which we deduce from general results by Bondarenko, Radchenko and Seip \cite{BRS}.

\end{abstract}

\maketitle
\section{Introduction}\label{sec:intro}
\subsection{Main result} 

The purpose of this paper is to prove the following interpolation formula, which generalizes to higher dimensions the ones obtained by Radchenko and Viazovska \cite{R-V}. Our main result is easiest to formulate in dimensions at least $5$. The theorem covering dimensions $2,3$ and $4$ will be stated and proved in  \S \ref{sec:small-dimensions}.

Throughout the paper, we write $\hat{f}$ for the Fourier transform of an integrable function $f$ on $\R^d$ (see \S \ref{sec:notation-fourier-transform} for the normalization). We often abbreviate by $S = S^{d-1}$ the unit sphere in $\R^d$ and always integrate with respect to probability surface measure over it.

\begin{theorem}\label{thm:main-thm-non-radial}
Let $d \geq 5$ and let the smooth functions $A_n, \tilde{A}_n : \R^d \times S^{d-1} \rightarrow \C$ be defined as in \S \ref{sec:proof-main-theorem}. Then, for every Schwartz function $f:\R^d \rightarrow \C$ and every point $x \in \R^d$, we have 
\begin{equation}\label{eq:ultimate-interpolation-formula-in-thm}
f(x) =\sum_{n=1}^{\infty}{\int_{S}{A_n(x, \zeta) f(\sqrt{n}\zeta) d\zeta}} +  \sum_{n=1}^{\infty}{\int_{S}{{\tilde{A}}_n(x, \zeta) \hat{f}(\sqrt{n}\zeta) d\zeta}}
\end{equation}
and both series converge absolutely. 
\end{theorem}

Formula \eqref{eq:ultimate-interpolation-formula-in-thm} holds more generally for functions $f$ on $\R^d$, such that $f$ and $\hat{f}$ decay sufficiently fast at infinity, see Corollary \ref{cor:extension}. In \S \ref{sec:proof-main-theorem}, we will show that the partial sums on the right hand side of \eqref{eq:ultimate-interpolation-formula-in-thm} converge uniformly, together with all partial derivatives, on compact subsets of $\R^{d} \setminus \{0 \}$.
\subsection{Context}\label{sec:intro-context}
Radchenko and Viazovska \cite{R-V} constructed a sequence of even Schwartz functions $a_n : \R \rightarrow \R$ such that, for all even Schwartz functions $f: \R \rightarrow \C$ and all points $x \in \R$, one has
\begin{equation}\label{eq:1-d-even-interpolation}
f(x) = \sum_{n=0}^{\infty}{f(\sqrt{n}) a_n(x)} + \sum_{n=0}^{\infty}{\hat{f}(\sqrt{n}) \widehat{a_n}(x)},
\end{equation}
where both series converge absolutely. They obtained a similar result for odd Schwartz functions, to which we will return shortly.  Their result is one out of a growing number of constructive existence theorems  in Euclidean harmonic analysis \cite{Viazovska8, CKMRV-24, R-V, CohnConcalves, CKMRV-E8-leech}, related to sphere packing, energy minimization, sign uncertainty principles and interpolation, in which the constructed object comes from a modular form or some generalization thereof. The employed method often involves a certain integral transform, which has first appeaerd in Viazovska's work on sphere packing \cite{Viazovska8}.

The cited works are primarily concerned with \emph{radial} Schwartz functions and focus on a particular dimension. The restriction to radial functions is natural given the type of questions that these works adress. In the context of the interpolation theorems \cite{R-V,CKMRV-E8-leech}, one can still ask if the formulas generalize to other dimensions, or to functions that are not necessarily radial. Theorem \ref{thm:main-thm-non-radial} provides such a generalization. 

We remark that, in contrast to the cited works on interpolation, we do not prove a  free interpolation theorem in our non-radial setting, so \eqref{eq:ultimate-interpolation-formula-in-thm} is merely a ``sampling formula". More precisely, we do not give sufficient conditions for when a pair of sequences of functions on the unit sphere comes from a Schwartz function, by restricting it and its Fourier transform  to the spheres $\sqrt{n}S^{d-1}$. We will provide some necessary conditions in \S \ref{sec:space-of-relations-infinite-dimensional}, which indicate that the problem is more difficult than in the radial setting.

We still have the following immediate corollary of Theorem \ref{thm:main-thm-non-radial}, which may be interesting in its own right. Theorem \ref{thm:non-radial-small-dimensions} below gives an analogous corollary for dimensions $2,3,4$, taking into account information near the origin.
\begin{corollary}\label{cor:trivial-corollary}
For $d \geq 5$, the only Schwartz function $f$ on $\R^d$ satisfying $f(\sqrt{n}S) = \hat{f}(\sqrt{n}S) = \{0 \}$ for all $n \geq 1$ is $f = 0$.
\end{corollary}
The corollary naturally extends to ellipsoids, by composing $f$ with an  invertible linear transformation and  correspondingly $\hat{f}$ with the adjoint of the inverse. One could ask whether a purely analytic proof of Corollary \ref{cor:trivial-corollary}  can be given, one that does not go through Theorem \ref{thm:main-thm-non-radial} or modular forms. Recently, J. Ramos and M. Sousa obtained Fourier uniqueness results in this direction, namely for radial Schwartz functions and sequences of interpolation nodes that concentrate ``\emph{more} densely near infinity than $\sqrt{n}$", see \cite[Thm. 1 and \S 5]{Ramos-Sousa} for precise statements. Our analysis in \S \ref{sec:harmonic-analysis} shows how one can deduce from such results corresponding uniqueness results for non-radial Schwartz functions, see Corollary \ref{cor:cheap-fourier-uniqueness}.
\subsection{Ideas}
We proceed by further outlining the contents of the paper and sketching some of the main ideas.
\subsubsection{Relationship to the radial problem} In  \S \ref{sec:harmonic-analysis}, we show how to deduce a non-radial interpolation formula  in dimension $d$, from radial ones in a sequence of higher dimensions. More specifically, we deduce such a formula from the existence of radial functions $a_{p,n}, \tilde{a}_{p,n}$ on $\R^{p}$, with $p \equiv d \pmod{2}$, having the property that for all $f \in \mathcal{S}_{\text{rad}}(\R^p)$ and all $x \in \R^p$,
\begin{equation}\label{eq:interpolation-radial-intro}
f(x) = \sum_{n=0}^{\infty}{ f(\sqrt{n})  a_{p,n}(x)} + \sum_{n=0}^{\infty}{  \hat{f}(\sqrt{n})  \tilde{a}_{p,n}(x)}.
\end{equation}
This step towards Theorem \ref{thm:main-thm-non-radial} is quite general and actually works for arbitrary sequences of interpolation nodes. It relies only on harmonic analysis on the sphere and Euclidean space with no reference to the particular nodes $\sqrt{n}$. We state the result as Corollary \ref{cor:cheap-non-radial}, which gives a formula that expresses any value $f(x)$ for $f \in \mathcal{S}(\R^d)$, as a sum of two double series of integrals over the sphere and may be badly behaved from the point of view of convergence. Interchanging sums and integrals formally, we  find candidates for the kernels $A_n, \tilde{A}_n$ in Theorem \ref{thm:main-thm-non-radial}.  To justify these rearrangments, we need absolute convergence of the double series and hence some specific bounds on the functions $a_{p,n}(x), \tilde{a}_{p,n}(x)$, that are explicit in all parameters, including the ``auxiliary" dimension $p$. We do not know, whether one can produce radial Schwartz functions obeying sufficient bounds, via contour integral methods similar to those in \cite{R-V, CKMRV-E8-leech}. We circumvent the use of such contour integrals by solving the radial interpolation problem in a different way, based on a method which is closely related to the construction of Poincar{\'e} series, described in a bit more detail below in \S \ref{sec:poincare-series-intro}.

A special case of using radial functions in higher dimensions, as mentioned above, is already implicit in Radchenko's and Viazovska's work, as we now briefly explain. Besides \eqref{eq:1-d-even-interpolation}, valid for even (i.e. radial) Schwartz functions on $\R$, Radchenko and Viazovska also find a formula for odd Schwartz functions. Since
\begin{equation}\label{eq:even-odd-decomposition}
\mathcal{S}(\R) = \mathcal{S}_{\text{even}}(\R) \oplus \mathcal{S}_{\text{odd}}(\R) =  \mathcal{S}_{\text{rad}}(\R) \oplus x\mathcal{S}_{\text{rad}}(\R),
\end{equation}
one can combine the two and write down a formula that reconstructs any $f \in \mathcal{S}(\R)$ from the values $f(\sqrt{n})$, $\hat{f}(\sqrt{n})$, together with the values $f'(0)$, $\hat{f}'(0)$. The underlying mechanism here is that the topological vector space $\mathcal{S}_{\text{odd}}(\R)$ is isomorphic to $\mathcal{S}_{\text{rad}}(\R^3)$, in a way that is compatible with Fourier transforms. To describe the isomorphism explicitly, let us define, for any $f \in \mathcal{S}(\R)$, the radial function $L f: \R^3 \rightarrow \C$, by 
\begin{equation}\label{eq:simple-definition-lift}
L f(x) = \frac{f(|x|)-f(-|x|)}{2|x|} \quad  \text{for } x  \in \R^3 \setminus \{0 \}\quad \text{and} \quad Lf(0) = f'(0).
\end{equation}
Using Taylor's Theorem one can show  that $Lf \in \mathcal{S}_{\text{rad}}(\R^3)$. In the other direction, we can define, for each $f \in \mathcal{S}_{\text{rad}}(\R^3)$, the Schwartz function $Rf  \in \mathcal{S}_{\text{odd}}(\R)$ by $Rf(t) = tf(t, 0, 0)$. The (continuous) linear maps $R$ and $L$ are then mutually inverse and intertwined with the Fourier transforms on $\R$ and $\R^3$ by $\widehat{Lf} = i L \hat{f}$, see \cite[\S 2.1 in ch.4]{Howe-Tan}. Thus, we can use use the map $R$ to ``transport" interpolation formulas as in \eqref{eq:interpolation-radial-intro} from $\mathcal{S}_{\text{rad}}(\R^3)$ to $\mathcal{S}_{\text{odd}}(\R)$.

For dimensions $d \geq 2$ we will define in \S \ref{sec:harmonic-analysis} a generalization of the map $L$ in \eqref{eq:simple-definition-lift}, by replacing the finite average of over the zero-dimensional sphere $S^{0} = \{-1, 1\}$  by a continuous average over $S^{d-1}$, one for each harmonic polynomial. In fact, the definitions can be written in the same way, by working with the probability measure on $S^{0}$, assigning mass $1/2$ to both of its endpoints. The finite direct sum \eqref{eq:even-odd-decomposition}  will be replaced by an infinite direct (topological)  sum, described by spaces of harmonic polynomials.
\subsubsection{Solving the radial problem by Poincar{\'e}-type series}\label{sec:poincare-series-intro}
To have a supply of radial functions satisfying \eqref{eq:interpolation-radial-intro}, we will prove in \S \ref{sec:poincare-type-series} the following Theorem. It pertains to dimensions at least 5. For dimensions $ 2,3,4$, we will need an additional result, that we  deduce from more general results by Bondarenko, Radchenko and Seip \cite{BRS}, see \S \ref{sec:small-dimensions}.
\begin{theorem}\label{thm:main-thm-radial-poincare}
Let $p \geq 5$. There exist sequences of even entire functions $b_{p,n}, \tilde{b}_{p,n} : \C \rightarrow \C$ such that, for every $f \in \mathcal{S}_{\text{rad}}(\R^p)$ and every $x \in \R^p$, we have
\begin{equation}\label{eq:interpolation-formula-radial-in-main-thm-poincare}
f(x) = \sum_{n=1}^{\infty}{f(\sqrt{n}) b_{p,n}(|x|)} + \sum_{n=1}^{\infty}{ \hat{f}(\sqrt{n}) \tilde{b}_{p,n}(|x|)}
\end{equation}
with absolute convergence. They obey the following bounds.
\begin{enumerate}[(i)]
\item There exist two constants $C_1, C_2 > 0$, independent of $p$, such that, for all $n \geq 1$, all $r \in \R$ and all $\varepsilon \in (0, 1/8]$, we have
\begin{align}
\max{(|b_{p,n}(r)|, |\tilde{b}_{p,n}(r)|) } &\leq C_1  (47/p)^{p/4} \, n^{p/2}, \label{eq:bounds-for-bpn-sup-norm-in-thm}\\
r \neq 0\; \Rightarrow \; \max{(|b_{p,n}(r)|, |\tilde{b}_{p,n}(r)| )} &\leq C_2 \varepsilon^{-2} n^{p/4+1+ \varepsilon}|r|^{-p/2+2(1+\varepsilon)}. \label{eq:bounds-for-bpn-nonzero-r-in-thm}
\end{align}
\item For every multi-index $\alpha \in \N_0^d$ and every $R>0$, there exist constants $C_3, C_4 >0$, depending on $d$, $\alpha$ and $R$, but not on $p$, such that for all $n \geq 1$, all $x \in \R^d$, with $|x| \leq R$ and all $\varepsilon \in (0, 1/8]$, we have
\begin{align}
\max{\left (|\partial^{\alpha}b_{p,n}(|x|)|, |\partial^{\alpha} \tilde{b}_{p,n}(|x|)|  \right)} &\leq C_3 (47 / p)^{p/4} n^{p/2+|\alpha|},  \label{eq:bound-partial-b_p-with-zero-in-thm} \\
 x \neq 0 \; \Rightarrow \; \max{\left (|\partial^{\alpha}b_{p,n}(|x|)|, |\partial^{\alpha} \tilde{b}_{p,n}(|x|)|  \right)} & \leq C_4 \varepsilon^{-2} n^{p/4+1+ \varepsilon+|\alpha|}|x|^{-p/2+2(1+ \varepsilon)}. \label{eq:bound-partial-b_p-away-from-zero-in-thm}
\end{align}
\end{enumerate}
\end{theorem}
\begin{remark-non}
The assertion in part (ii) includes implicitly that for each $d \in \N$, the functions $x \mapsto b_{p,n}(|x|), \tilde{b}_{p,n}(|x|)$ are smooth on $\R^d$, in particular in a neighborhood of the origin. The number $47$ comes from $47 \geq 2\pi e^2 \approx 46.4$. 
\end{remark-non}
We now briefly explain what goes into the proof of Theorem \ref{thm:main-thm-radial-poincare}. Let $\H = \{\tau \in \C\,:\, \imag(\tau) > 0 \}$ denote the upper half plane. The strategy is to find the generating functions \[
F_p(\tau,r) = \sum_{n=1}^{\infty}{b_{p,n}(r)e^{\pi i n \tau}}, \qquad \tilde{F}_p(\tau,r) = \sum_{n=1}^{\infty}{\tilde{b}_{p,n}(r) e^{\pi i n\tau }},
\]
knowing only that they need to satisfy a certain functional equation, which comes from applying the desired interpolation formula \eqref{eq:interpolation-formula-radial-in-main-thm-poincare} to Gaussians $e^{\pi i \tau r^2}$. This strategy has already appeared in \cite{R-V, CKMRV-E8-leech} and we explain the version we need in \S \ref{sec:generating-series-and-functional-equations}. The cited works succeed in finding the generating functions by integrating a suitable meromorphic and separately modular kernel function on $\H \times \H$ against the Gaussian $e^{\pi i z r^2}$ over a suitable path. Here we use a different method, which is closely related to the construction of Poincar{\'e} series and partly inspired by the works of Knopp on Eichler cohomology \cite{Knopp-eichler-coh}. 

In the context of classical modular forms, a Poincar{\'e} series $P_m$ has an \emph{integral} parameter $m \geq 1$, which indicates that the $m$th Fourier coefficient of a cusp form is returned when we pair it against $P_m$ with respect to the Petersson inner product. It is constructed by averaging the function $e^{\pi im \tau}$, with respect to the so-called slash-action, over cosets of the subgroup of translations, of the congruence subgroup involved. In our case, the relevant congruence subgroup is $\Gamma(2)$. Roughly speaking, we will modify this construction by summing over a specific subset of $\Gamma(2)$, which represents the above coset space, up to the identity coset and, instead of averaging the function $e^{\pi i m \tau} = e^{\pi i \sqrt{m}^2 \tau}$, we will average the Gaussian $e^{\pi i r^2 \tau}$ over that subset (for any $r \in \C$), so that, when $r^2 \in \Z$, we almost  have $F_p(\tau,r) = P_{r^2}(\tau)$, up to the constant term in the Fourier expansion and up to constant multiples.

By imitating the classical computation for the Fourier coefficients of Poincar{\'e} series, we can write $b_{p,n}(r)$ as an infinite series involving Bessel functions and finite exponential sums sums that look very much like classical Kloosterman sums, see  \eqref{eq:bpn-explicit-at-non-zero}, \eqref{eq:bpn-explicit-at-zero}.  By specializing these formulas to $r = \sqrt{m}$ and even dimensions $p \geq 6$, we will see that, if $n \neq m$, the value $b_{p,n}(\sqrt{m})$ equals (up to constant factors) the $n$th Fourier coefficient of the $m$th Poincar{\'e}-series in weight $p/2$ with respect to  $\Gamma_0(4)$ (which is conjugate to $\Gamma(2)$) and character $\chi^{k}$, where $\chi$ is the non-trivial Dirichlet character  modulo $4$. These observations allow us to deduce that, for infinitely many indices $n$, the function $r \mapsto |r|^{ p/2-1+\varepsilon}|b_{p,n}(r)|$ is unbounded on $\R$, for every $\varepsilon > 0$, see Proposition \ref{prop:lower-bound-bpn}. In particular, infinitely of the functions $b_{p,n}(r)$ are not of rapid decay on $\R$.

\subsection{General notation and a few preliminary facts}\label{sec:preliminaries}
\subsubsection{Radial functions}\label{sec:notation-radial-near-origin}
A function $f$ on $\R^d$ is \emph{radial}, if $f(x) = f(y)$ for all vectors $x,y \in \R^d$ with the same Euclidean norm $|x| = |y|$. If $f$ is radial and $r \geq 0$ is a real number, we will sometimes abuse notation and denote also by $f(r)$ the common value of $f$ on the set $rS^{d-1} =\{x \in \R^d\,:\, |x|=r\}$. 

We denote by $\mathcal{S}(\R^d)$ the Schwartz space and by $\mathcal{S}_{\text{rad}}(\R^d)$ the subspace of radial Schwartz functions. We use the standard topology on these spaces. For later reference, we record the following  convenient lemma, which follows from Proposition 3.3 in \cite{Grafakos}. 
\begin{lemma}\label{lem:radial-extension-of-even}
For every $p \geq 1$, the assignment $f \mapsto (x \mapsto f(|x|))$ defines a continuous linear map $\mathcal{S}_{\text{rad}}(\R) \rightarrow \mathcal{S}_{\text{rad}}(\R^p)$.
\end{lemma} 
The proof of Proposition 3.3 in \cite{Grafakos} uses an old result of Hassler Whitney \cite{Whitney}, asserting that for every smooth even function $\phi: \R \rightarrow \C$ there exists a smooth function $w : \R \rightarrow \C$ such that $\phi(r) = w(r^2)$ for all $r \in \R$. As a consequence, we see that  for every $p \geq 1$, the assignment $\phi \mapsto (x \mapsto \phi(|x|))$ gives a well-defined linear map $C_{\text{rad}}^{\infty}(\R) \rightarrow C_{\text{rad}}^{\infty}(\R^p)$.
\subsubsection{Fourier transforms}\label{sec:notation-fourier-transform}
Given an integrable function $f: \R^d \rightarrow \C$ we denote by $ \mathcal{F}(f) = \hat{f}$  its Fourier transform, which we normalize by $\hat{f}(\xi) = \int_{\R^d}{f(x) e^{-2 \pi i x \cdot \xi}dx}$, where $x \cdot \xi$ denotes the Euclidean inner product of $x, \xi \in \R^d$. We will sometimes compare the Fourier transform of functions on $\R^d$ and radial functions on $\R^{d+2m}$, but context and notation should make it clear in which dimension the Fourier transform is computed.
\subsubsection{Square roots}\label{sec:notation-square-roots}
We denote by $\H = \{z \in \C \,:\, \imag(z) > 0\}$ the complex upper half plane. Given $k \in \C$, we define $(- i \tau)^{k} = (\tau/i)^k = \exp{( k \log{(\tau/i)})}$, where we choose the holomorphic function $\tau \mapsto \log(\tau/i)$ in such a way that it is value at $ \tau = i$ is $0$.
\subsubsection{Two-periodic holomorphic functions}\label{sec:two-periodic-hol-functions}
We denote the open unit disc by $\mathbb{D} = \{ w \in \C\,:\,|w| < 1 \}$ and by $\mathbb{D}^{\times}= \mathbb{D} \setminus \{ 0 \}$ the punctured open unit disc. Given a two-periodic holomorphic function $F : \H \rightarrow \C$, write $F_{\text{disc}}: \mathbb{D}^{\times} \rightarrow \C$ for the unique holomorphic function satisfying $F_{\text{disc}}(e^{\pi i z}) = F(z)$ for all $z \in\H$. Then $F$ admits a Fourier--Laurent expansion $F(z) = \sum_{n \in \Z}{\widehat{F}(n) e^{\pi i n z}}$ with Fourier--Laurent coefficients given by 
\[
\widehat{F}(n) = \frac{1}{2}\int_{iy_0 + [-1,1]}{F(x + iy_0) e^{- \pi in x}dx} = \frac{1}{2\pi i}\int_{|w|= \delta}{F_{\text{disc}}(w) \frac{dw}{w^{n+1}}},
\]
for any $y_0 > 0$ and any $\delta \in (0,1)$. We say that $F$ is meromorphic (holomorphic, vanishes) \emph{at infinity} if $F_{\text{disc}}$ is meromorphic  (holomorphic, vanishes) at zero.

\subsubsection{Gaussians}\label{sec:notation-gaussians}
For $p \geq 1$ and $z \in \H$ we denote by $G_p(z) \in \mathcal{S}_{\text{rad}}(\R^p)$ the function defined by $G_p(z)(x) = G_p(z,x)= e^{\pi i z |x|^2}$ for $x \in \R^p$ and we refer to it as the \emph{Gaussian} (with parameter $z$). A proof of the following important Lemma can be found in \cite[Lemma 2.2]{CKMRV-E8-leech} and will be used in the proof of Proposition \ref{prop:lifts} and in \S \ref{sec:generating-series-and-functional-equations}.
\begin{lemma}\label{lem:density-gaussians-and-continuity}
The set $\{G_p(z)\,:\, z \in \H\}$ spans a dense subspace of $\mathcal{S}_{\text{rad}}(\R^p)$.
\end{lemma}

\section{Harmonic analysis part}\label{sec:harmonic-analysis}
The goal of this section to write down an interpolation formula for Schwartz functions on $\R^d$, assuming that one has interpolation formulas for \emph{radial} Schwartz functions in every dimension $p \in \{d +2m \,:\, m \in \N_0\}$.

To fix notation, we first recall some basic definitions and facts about harmonic polynomials and spherical harmonics. All of these facts can be found \cite[ch. 3]{Stein-Weiss} and \cite[ch. 5]{AxlerBourdonRamey}. Let $d \geq 2$. For each $m \in \N_0$, let $\mathcal{H}_m(\R^d)$ denote the space of all complex-valued harmonic polynomial functions on $\R^d$, which are homogeneous of degree $m$. We call these harmonic polynomials (of degree $m$) for short. Let $\mathcal{H}_m(S^{d-1})$ denote the space of all restrictions $u|_{S^{d-1}}$ of $u  \in \mathcal{H}_m(\R^d)$. It is the space of spherical harmonics of eigenvalue $-m(d-2+m)$ for the spherical Laplacian and carries an $L^2$-inner product structure, coming from the probability surface measure on $S^{d-1}$. Via restriction, the spaces $\mathcal{H}_m(\R^d)$ and $\mathcal{H}_m(S^{d-1})$ are by definition isomorphic and we will freely use this isomorphism to give meaning to ``orthonormal basis" $\mathcal{B}_m \subset \mathcal{H}_m(\R^d)$ or to make sense of values $u(x)$ for $x \in \R^d$, even when $u$ was initially declared to belong to $\mathcal{H}_m(S^{d-1})$. We have
\begin{equation}\label{eq:dimension-formula-harmonic-poly}
\dim_{\C}{(\mathcal{H}_m(\R^d))} = \binom{d+m-1}{d-1}- \binom{d+m-3}{d-1} \sim  \frac{2}{(d-2)!} m^{d-2},
\end{equation}
as $m \rightarrow \infty$. For each point $\omega \in S^{d-1}$ and each $m \in \N_0$, let $\zeta \mapsto Z_m^d(\zeta, \omega)$ denote the zonal spherical harmonic of degree $m$ with pole $\omega$, characterized by the property  
\begin{equation}\label{eq:defining-property-zonal}
\int_{S^{d-1}}{u(\zeta) \overline{Z_m^d(\zeta, \omega)} d\zeta} = u(\omega) \qquad \text{for all} \quad u \in \mathcal{H}_m(S^{d-1}).
\end{equation}
For \emph{any} orthonormal basis  $\mathcal{B}_m \subset \mathcal{H}_{m}(S^{d-1})$, we have
\begin{equation}\label{eq:zonal-spherical-in-terms-of-basis}
Z_{m}^d(\zeta, \omega) = \sum_{u \in \mathcal{B}_m}{u(\zeta) \overline{u(\omega)}}
\end{equation}
and for each $\omega \in S^{d-1}$,  one has 
\begin{equation}\label{eq:zonal-on-diagonal}
Z_m^d(\omega, \omega) = \| Z_{m}^d( \cdot, \omega)\|_{L^2(S^{d-1})}^2 = \dim{\mathcal{H}_m(\R^d)}.
\end{equation}
It follows from \eqref{eq:defining-property-zonal}, \eqref{eq:zonal-on-diagonal} and the Cauchy--Schwarz inequality that 
\begin{equation}\label{eq:sup-norm-dimension-bound}
\sup_{\zeta \in S^{d-1}}{|u(\zeta)|} \leq \| u \|_{L^2(S^{d-1})} \left( \dim{\mathcal{H}_m(\R^d)} \right)^{1/2},
\end{equation}
for every $u \in \mathcal{H}_m(S^{d-1})$. We will also use the fact that every homogeneous polynomial  $P : \R^d \rightarrow \C$ of degree $m$ can be (uniquely) written as 
\begin{equation}\label{eq:decompositoin-general-poly-into-harmonics}
P(x) = \sum_{0 \leq j \leq m/2}{|x|^{2j}u_{j}(x)}, \quad \text{for some } \quad u_{j} \in \mathcal{H}_{m-2j}(\R^d).
\end{equation} 
The next definition and proposition will generalize the discussion surrounding \eqref{eq:simple-definition-lift} in the introduction. For any fixed $u \in \mathcal{H}_m(\R^d)$, we give here an explicit inverse  of the natural map $\mathcal{S}_{\text{rad}}(\R^{d+2m}) \rightarrow u \mathcal{S}_{\text{rad}}(\R^d)$ (up to constant multiples), which intertwines the Weil representations of a two-fold covering group of $\SL_2(\R)$ acting on the respective Schwartz spaces (see \cite[Ch. 3]{Howe-Tan}). This is closely related to  Bochner's periodicity relations and the transformation laws for harmonic theta series, see  \cite[Ch. 3, Ch. 4]{Howe-Tan} and \cite{bochner-theta-relations}. The result may be known in some equivalent form, but we include our proof to keep the presentation self-contained.
\begin{definition}\label{def:lifts}
Let $d \geq 2$, $m \in \N_0$ and $u \in \mathcal{H}_m(\R^d)$.  For each $f \in C^{\infty}(\R^d)$ and each $p \in \N$ we define the radial function $L_{u}^p f : \R^p \rightarrow \C$ by
\begin{align*}
L_{u}^{p}f(x) &= \int_{S^{d-1}}{f(|x|\zeta) \overline{u(\zeta/|x|)}d\zeta} \quad  \text{ for  }x  \in\R^{p} \setminus \{0\},\\
  L_{u}^{p}f(0) &= \sum_{|\alpha|=m}{\frac{(\partial^{\alpha}f)(0)}{\alpha!}\int_{S^{d-1}}{\zeta^{\alpha}\overline{u(\zeta)}d\zeta}}.
\end{align*}
\end{definition}

\begin{proposition}\label{prop:lifts}
With notations as in Definition \ref{def:lifts}, the following holds.
\begin{enumerate}[(i)]
\item Each $L_u^{p}f$ is a smooth radial function on $\R^p$.
\item The assignment $f \mapsto L_u^{p}f$ defines a continuous linear map $\mathcal{S}(\R^{d}) \rightarrow \mathcal{S}_{\text{rad}}(\R^p)$.
\item For all $f \in \mathcal{S}(\R^d)$ we have $\mathcal{F}(L_{u}^{d+2m}f) =  i^m L_u^{d+2m}\mathcal{F}(f)$.
\end{enumerate}
\end{proposition}
\begin{proof}
%
Fix $d \geq 2$, $m \geq 0$, $f \in C^{\infty}(\R^d)$ and $u \in \mathcal{H}_m(\R^d)$. We prove parts (i) and (ii) in the case $p = 1$, which will imply the general case by the discussion in \S \ref{sec:notation-radial-near-origin}. We therefore temporarily write $Lf(y) = L_u^1 f(y)$ for $y \in \R$. To start, recall that by Taylor's theorem we have, for every $x \in \R^d$ and every $K \in \N_0$,  
\[
f(x) = \sum_{k=0}^{K}{ \sum_{|\alpha|=k}{ \frac{(\partial^{\alpha} f)(0)}{\alpha!} x^{\alpha}}} + \sum_{|\alpha|=K+1}{\frac{K+1}{\alpha!} \int_{0}^{1}{(1-t)^K (\partial^{\alpha} f) (tx)  dt} \,  x^{\alpha}}.
\]
We specialize this to $x = |y|\zeta$, where $(y, \zeta) \in \R^{\times} \times S$ and take $K \geq m+1$. Then we integrate over $\zeta \in S$ against $\overline{u}(\zeta/ |y|)$ and use the decomposition \eqref{eq:decompositoin-general-poly-into-harmonics}, applied to monomials $P(x) =x^{\alpha}$, combined with orthogonality relations for spherical  harmonics, to obtain
\begin{equation}\label{eq:expression-lift-by-taylor}
Lf(y) = \sum_{\substack{k =m \\ k \equiv m (2)}}^{K}{|y|^{k-m} \sum_{|\alpha|=k}{ \frac{(\partial^{\alpha}f)(0)}{\alpha!} \int_{S}{\zeta^{\alpha} \overline{u(\zeta)} d\zeta}  }} + |y|^{K+1-m}R_K(y),
\end{equation}
with remainder term
\[
R_K(y)= \sum_{|\alpha|=K+1}{ \frac{K+1}{\alpha!} \int_{S}{ \int_{0}^{1}{      (1-t)^{K} (\partial^{\alpha} f)( |y| \zeta t) dt\, \overline{u(\zeta)} \zeta^{\alpha}  d\zeta}}}.
\]
The  first sum in \eqref{eq:expression-lift-by-taylor} is a polynomial in $y^2$, hence in $C_{\text{rad}}^{\infty}(\R)$. It therefore suffices to show that $y \mapsto |y|^{K+1-m}R_K(y)$ belongs to $C^{\ell(K)}(\R)$ in such a way that $\ell(K) \rightarrow \infty$ as $K \rightarrow \infty$.  To that end, we first check that on $\R^{\times}$, we have
\begin{align}
\frac{d^{j}}{d y^{j}}|y|^{c} &= (y/|y|)^{j} \frac{c!}{(c-j)!}|y|^{c-j} \qquad (0 \leq j \leq c), \label{eq:aux-derivative-formula-1}\\
\frac{d^{j}}{d y^{j}}  (\partial^{\alpha}f)(t\zeta |y|) &= t^{j}(y/|y|)^{j}\sum_{|\beta|= j}{(\partial^{\alpha+ \beta} f)(|y|t \zeta) \zeta^{\beta}}. \label{eq:aux-derivative-formula-2}
\end{align}
%
%
We now take $K$ of the form $K=m + 2N$ for $N \in \N$. Then we deduce from the Leibniz rule and the above formulas \eqref{eq:aux-derivative-formula-1}, \eqref{eq:aux-derivative-formula-2} that, for  $0 \leq j \leq N$, the derivative $ \frac{d^{j}}{dy^{j}}|y|^{K-m+1} R_K(y)$ is equal to $(y/|y|)^{j}= (y/|y|)^{j_1}(y/|y|)^{j_2}$, times
\begin{equation}\label{eq:derivative-remainder-near-origin}
 \sum_{j_1 +j_2 = j}{ a_{j_1,j_2} |y|^{2N+1-j_1}    \sum_{\substack{|\alpha| =K+1\\|\beta|=j_2} }{ \frac{K+1}{\alpha!} \int_{S}{ \int_{0}^{1}{t^{j_2}(1-t)^{K} (\partial^{\alpha + \beta}f)(t|y|\zeta)}       dt}\,\zeta^{\alpha + \beta} \overline{u(\zeta)} d \zeta} },
 \end{equation}
where $a_{j_1,j_2} = \tfrac{j!}{j_1! j_2 !} \tfrac{(2N+1)!}{ (2N+1-j_1)!}$. All of these computations hold for $y \in \R^{\times}$. We deduce that  $ \frac{d^{j}}{dy^{j}}|y|^{K-m+1} R_K(y) \rightarrow 0$, as $y \rightarrow 0$ on $\R^{\times}$ and that the relevant difference quotients at $y=0$ also tend to zero.

We now turn to part (ii), so assume that $f \in \mathcal{S}(\R^d)$ and still that $p = 1$. Fix integers $j,n \geq 0$ such that $n$ is even. Define
\[
A = \sup_{y \in [0,1]}{|(1+y^n) (Lf)^{(j)}(y)|}, \qquad B=\sup_{y \in [1,\infty)}{|(1+y^n) (Lf)^{(j)}(y)|}.
\]
It suffices to show that $A$ and $B$ can be bounded in terms of finitely many continuous semi-norms of $f$. Here, we also used that $(Lf)^{(j)}$ is either even or odd, to be able to restrict to non-negative arguments $y$, for convenience.

To estimate the term $A$, we again take $K = 2N+m$ with $j \leq N$. We then read off from \eqref{eq:expression-lift-by-taylor} that the $j$th derivative of the  polynomial $Lf(y)-|y|^{2N+1}R_K(y)$ has degree at most $2N- j$, and that its coefficients  are multiples of $\partial^{\alpha} f(0)$, with $|\alpha| \leq K$, so that the supremum over $y \in [0,1]$ of that derivative may be bounded in terms of finitely many continuous semi-norms of $f$. For the remainder term we note that inside the integrals appearing in \eqref{eq:derivative-remainder-near-origin}, the vectors $t|y| \zeta \in \R^d$ have Euclidean norm at most $1$ for all triples $(t,y, \zeta) \in [0,1]^2 \times S$ under consideration, so that we can bound these integrals in terms of suprema of partial derivatives of $f$, over the closed unit ball in $\R^d$.

To estimate the term $B$, we compute directly from the definition, using the Leibniz rule as well as \eqref{eq:aux-derivative-formula-2} (with $\alpha =0, t = 1$), that, for $m \geq 1$, $y \geq 1$,
\begin{equation}\label{eq:derivative-lift-away-from-origin}
(Lf)^{(j)}(y) = \sum_{j_1 +j_2 = j}{ b_{j, j_1,j_2} y^{-m-j_1}\sum_{|\beta|=j_2}{\int_{S}{(\partial^{\beta}f)(y \zeta) \zeta^{\beta} \overline{u(\zeta)} d\zeta}}},
\end{equation}
where $b_{j, j_1,j_2} =\tfrac{j!}{j_1! j_2!} \tfrac{ (-1)^{j_1} (m+j_1-1)!}{ (m-1)!}$. If $m = 0$, the formula for $(Lf)^{(j)}$ is simpler (namely only the inner sum in \eqref{eq:derivative-lift-away-from-origin} with $j_2$ replaced by $j$ and $\overline{u(\zeta)}$ replaced by $1$). We may now multiply \eqref{eq:derivative-lift-away-from-origin} with $1+y^{n}$, and use that 
\[
|(1+y^{n})(\partial^{\beta}f)(y \zeta)| \leq \sup_{ |x| \geq 1}{(1+|x|^n)|\partial^{\beta}f(x)|},
\]
using $y^n = |y \zeta|^n$ for $\zeta \in S = S^{d-1}$ for the inequality here. Thus, $B$ can be bounded in terms of $f$ as required.

We turn to part (iii) in which we assume that $p = d+2m$ and that $f \in \mathcal{S}(\R^d)$. By part (ii) and continuity of the Fourier transform, we may assume that $f$ belongs to a (generating set of a) dense subspace of $\mathcal{S}(\R^d)$. It thus suffices to consider Schwartz functions $f$ of the form  $f(x) = u_0(x) e^{\pi i z|x|^2}$, for some $u_0 \in \mathcal{H}_{m_0}(\R^d)$, $m_0 \in \N_0$ and $z \in \H$, because:
\begin{itemize}
\item the linear span of all Schwartz functions of the form $x \mapsto P(x) e^{- \pi |x|^2}$, where $P: \R^d \rightarrow \C$ is a polynomial function, is dense in $\mathcal{S}(\R^d)$, see \cite[Ch.3, Ex. 6]{Howe-Tan},
\item by \eqref{eq:decompositoin-general-poly-into-harmonics}, every polynomial $P$ on $\R^d$, is a sum of products of a harmonic polynomial with an even power of the Euclidean norm,
\item as the parameter $z$ traverses the upper half plane $\H$, the Gaussians $e^{\pi i z|x|^2}$ span a dense subspace of  $\mathcal{S}_{\text{rad}}(\R^d)$, see Lemma \ref{lem:density-gaussians-and-continuity}.
\end{itemize}
Under this assumption on $f$, we have, by definition,
\[
L_u^{d+2m}f (y) = \int_{S}{e^{\pi i z \|y| \zeta|^2} u_0(|y|\zeta) \overline{u(\zeta/|y|)} d \zeta} = e^{\pi i z |y|^2} |y|^{m_0-m} \langle u_0,u  \rangle_{L^2(S)},
\]
for all $y \in \R^{d+2m} \setminus \{0 \}$. If $m_0 \neq m$, then  $\langle u_0,u  \rangle_{L^2(S)} =0$, by orthogonality. If $m_0 = m$, then
\begin{equation}\label{eq:fourier-transform-of-lift}
\mathcal{F}(L_u^{d+2m}f)(\eta) = (-i z)^{-\frac{d+2m}{2}} e^{\pi i (-1/z) |\eta|^2} \langle u_0, u \rangle_{L^2(S)},
\end{equation}
for every $\eta \in \R^{d+2m}$. On the other hand, the Hecke-Funk identity, which follows from \cite[Thm 3.4]{Stein-Weiss} by homogeneity and analyticity, says that for all $\xi \in \R^d$, one has
\[
\hat{f}(\xi) = (-i)^{m_0} (-i z)^{-\frac{d+2m_0}{2}} u_0(\xi) e^{\pi i (-1/z) |\xi|^2}.
\]
From Definition \ref{def:lifts} we see 
\begin{align}
(L_u^{d+2m}\hat{f})(\eta) &= (-i)^{m_0} (-i z)^{-\frac{d+2m_0}{2}} \int_{S}{e^{\pi i (-1/z)||\eta| \zeta|^2} u_0(|\eta|\zeta) \overline{u(\zeta/|\eta|)} d \zeta} \notag \\
&= (-i)^{m_0} (-i z)^{-\frac{d+2m_0}{2}} e^{\pi i z |\eta|^2} |\eta|^{m_0-m} \langle u_0, u \rangle_{L^2(S)}, \label{eq:lift-of-fourier-transform}
\end{align}
for every $\eta \in \R^{d+2m} \setminus \{0 \}$. If $m_0 \neq m$, then this again is zero. Otherwise, by comparing \eqref{eq:fourier-transform-of-lift} with \eqref{eq:lift-of-fourier-transform} we obtain the formula claimed in (iii). 
\end{proof}
%
\begin{corollary}\label{cor:cheap-non-radial}
Let $d \geq 2$. Let $(r_n)_{n \in \N_0}$, $(\rho_n)_{n \in \N_0}$ be two sequences of non-negative real numbers. Suppose we are given, for each integer $p \in \{d +2m \,:\, m \in \N_0 \}$, each real number $r \geq 0$ and each $n \in \N_0$, two complex numbers $c_{p,n}(r), \tilde{c}_{p,n}(r)$ such that: for all $g \in \mathcal{S}_{\text{rad}}(\R^p)$ and all $y \in \R^p$, 
\[
g(y) = \sum_{n=0}^{\infty}{c_{p,n}(|y|) g(r_n)} + \sum_{n=0}^{\infty}{\tilde{c}_{p,n}(|y|) \hat{g}(\rho_n)},
\] 
and both of these series converge (not necessarily absolutely). Then, for every $x \in \R^d$ and every $f \in \mathcal{S}(\R^d)$,
\begin{align}
f(x) = &\sum_{m=0}^{\infty}{  \sum_{n=0}^{\infty}{ \Big(  c_{d+2m,n}(|x|) \int_{S}{f(r_n \zeta) Z_{m}^d(x, \zeta/ r_n) d\zeta} }} \notag \\  
& \qquad \qquad \quad  + i^m \tilde{c}_{d+2m,n}(|x|) \int_{S}{\hat{f}(\rho_n\zeta) Z_{m}^d(x, \zeta/ \rho_n) d\zeta} \Big) ,  \label{eq:pre-interpolation-formula}
\end{align}
where, if $\rho_n=0$ or $r_n = 0$, the integrals are defined through Definition \ref{def:lifts}. The double series converges in the indicated order of summation and is such that $\sum_{m =0}^{\infty}{|(\dots)|} < \infty$.
\end{corollary}
\begin{proof}
For every $m \geq 0$ we choose an orthonormal basis $\mathcal{B}_m \subset \mathcal{H}_m(S^{d-1})$ and we let $f \in \mathcal{S}(\R^d)$. Then for every $r \geq 0$, the function  $\omega \mapsto f(r\omega)$ is smooth on $S^{d-1}$, so that its $L^2$-expansion into spherical harmonics 
\begin{equation}\label{eq:general-point-wise-spectral-expansion-on-sphere}
f(r \omega) =  \sum_{m=0}^{\infty}{ \sum_{u \in \mathcal{B}_m}{ u(\omega) \int_{S}{f(r \zeta)  \overline{u(\zeta)}  d\zeta} }}
\end{equation}
converges pointwise absolutely and uniformly with respect to the sup-norm. Now let $x \in \R^{d} \setminus \{0\}$. In this proof, we write $\iota_m(x) = (x,0) \in \R^{d+2m}$ for the vector whose first $d$ coordinates are given by those of $x$ and whose last $2m$ coordinates are all zero. Since  \eqref{eq:general-point-wise-spectral-expansion-on-sphere} holds for $r =|x|$ and $\omega=x/|x|$ and since each $u \in \mathcal{B}_m$ is homogeneous of degree $m$, we obtain 
\begin{equation}\label{eq:general-expansion-of-f-via-lifts}
f(x)  =  \sum_{m=0}^{\infty}{  \sum_{u \in \mathcal{B}_m}{u(x) \int_{S}{f(|x| \zeta) \overline{u(\zeta/|x|)}  d\zeta}}} =\sum_{m=0}^{\infty}{  \sum_{u \in \mathcal{B}_m}{u(x) L_u^{d+2m}f(\iota_m(x))}},
\end{equation}
using Definition \ref{def:lifts}. Here, we could have embedded the vector $x$ also in any other space $\R^{p(m)}$ and  \eqref{eq:general-expansion-of-f-via-lifts} would be true with $L_u^{d+2m}f$ replaced by $L_u^{p(m)}f$. The point is that $p(m) = d+2m$ allows us to use part (iii) of Proposition \ref{prop:lifts}  and the assumption, giving
\begin{equation}\label{eq:interpolation-formula-radial-applied-to-lift}
 L_u^{d+2m}f(\iota_m(x)) = \sum_{n=0}^{\infty}{ \left( c_{d+2m,n}(|x|)L_u^{d+2m}f(r_n) +    \tilde{c}_{d+2m,n}(|x|)i^m L_u^{d+2m}\hat{f}(\rho_n)  \right)  }.
\end{equation}
Inserting \eqref{eq:interpolation-formula-radial-applied-to-lift} back into \eqref{eq:general-expansion-of-f-via-lifts} gives \eqref{eq:pre-interpolation-formula} (by recalling \eqref{eq:zonal-spherical-in-terms-of-basis}). As we assumed that $x \neq 0$, we still need to show that
\begin{align*}
f(0) &= \sum_{n=0}^{\infty}{ \left( c_{d,n}(0)\int_{S}{f(r_n \zeta) d\zeta  } + \tilde{c}_{d,n}(0)\int_{S}{\hat{f}(\rho_n \zeta) d\zeta  } \right)  }\\
&= \sum_{n=0}^{  \infty}{\left( c_{d,n}(0)L_1^{d}f(r_n) + \tilde{c}_{d,n}(0) \mathcal{F}(L_{1}^d f)(\rho_n) \right) },
\end{align*}
where $1$ stands for the constant polynomial $1$. But this identity holds by the assumed radial interpolation formula, applied to $L_1^d(f) \in \mathcal{S}_{\text{rad}}(\R^d)$ at the point zero.
\end{proof}
We record a further corollary of the general expansion in \eqref{eq:general-expansion-of-f-via-lifts} and part (iii) of Proposition \ref{prop:lifts}. It allows one to translate Fourier uniqueness results for radial functions in all dimensions, to corresponding uniqueness results for non-radial functions in a fixed dimension. It may be applicable to the generalization of the uniqueness results by J. Ramos and M. Sousa \cite{Ramos-Sousa} to radial functions in higher dimensions, as sketched in \S 5 of their paper. The statement of the corollary itself will not be used elsewhere in the paper, but might be relevant for future work. 
\begin{corollary}\label{cor:cheap-fourier-uniqueness}
Fix a dimension $d \geq 2$ and fix two subsets $R, \hat{R} \subset (0, \infty)$. Suppose that for all $p \in \{d + 2m \,:\, m \in \N_0 \}$ an all $f \in \mathcal{S}_{\text{rad}}(\R^p)$, the following implication holds
\begin{equation}\label{eq:uniquness-implicatoin}
 \left( f|_{\bigcup_{r \in R}{r S^{p-1}}} = 0 \; \text{ and } \; \hat{f}|_{\bigcup_{\rho \in \hat{R}}{\rho S^{p-1}}}  = 0 \right) \quad  \Longrightarrow \quad f = 0.
\end{equation}
Then the same implication holds for arbitrary $f \in \mathcal{S}(\R^p)$.
\end{corollary}
\begin{proof}
Suppose that $f \in \mathcal{S}(\R^d)$ vanishes on all spheres $r S^{d-1}$, $r \in R$ and that $\hat{f}$ vanishes on all spheres $\rho S^{d-1}$, $\rho \in \hat{R}$. Fix a nonzero point $x \in \R^d$ and aim to show that $f(x) = 0$ (which suffices by continuity). By \eqref{eq:general-expansion-of-f-via-lifts}, it suffices to show that for all $m \geq 0$ and $u \in \mathcal{H}_m(\R^d)$, the function $L_u^{d+2m}f \in \mathcal{S}_{\text{rad}}(\R^{d+2m})$ and its Fourier transform $i^m L_u^{d+2m}\hat{f}$ (using part (iii) of Proposition \ref{prop:lifts} here), vanish identically. By the assumption \eqref{eq:uniquness-implicatoin}, this is implied by the vanishing of these radial functions at all radii $r \in R$ and $\rho \in \hat{R}$ respectively. That in turn, follows directly from the definition of $L_{u}^{p}$ and our assumption on $f$.
\end{proof}

We conclude section \ref{sec:harmonic-analysis} with the following lemma giving bounds for the $L^2$-norm of derivatives of harmonic polynomials. It will be used in in the proof of Lemma \ref{lem:A-inf-summands-bounds} below.
\begin{lemma}\label{lem:derivative-harmonic-polynomial}
Let $d \geq 2$, $m \geq 0$ and $\gamma \in \N_0^d$ and assume $(m,\gamma) \neq (0,0)$. Set $c=|\gamma|$. Then, for all $u \in \mathcal{H}_m(\R^d)$, we have
 \[
\|\partial^{\gamma} u\|_{L^2(S)} \leq \sqrt{d^c}\, m^{c} \|u\|_{L^2(S)}.
\] 
\end{lemma}
\begin{proof}
We may assume that $m \geq 1$ and that $c \leq m$, as otherwise $\partial^{\gamma}u = 0$. By \cite[Thm 5.14]{AxlerBourdonRamey} there exists a constant $\nu_d  > 0$ so that for all $u, v \in \mathcal{H}_m(\R^d)$ of the form $u(x) = \sum_{|\alpha| = m}{b_{\alpha} x^{\alpha}}$, $v(x) = \sum_{|\alpha| = m}{ c_{\alpha} x^{\alpha}}$, we have
\[
\langle u, v \rangle_{L^2(S)} = \int_{S}{u(\zeta) \overline{v(\zeta)}d \zeta} = \nu_d \prod_{i=0}^{m-1}{(d+2i)^{-1}} \sum_{|\alpha|=m}{\alpha! b_{\alpha} \overline{c_{\alpha} } }.    
\]
Applying this with $u = v$ and computing $\partial^{\gamma} u(x) = \sum_{|\alpha|=m, \alpha \geq \gamma}{c_{\alpha} \frac{\alpha!}{(\alpha-\gamma)!} x^{\alpha- \gamma}}$, we obtain
	\[
\|  \partial^{\gamma} u\|_{L^2(S)}^2 \leq \left( \prod_{i=m-c}^{m-1}{(d+2i)} \right) \left(\max_{ \substack{|\alpha| =m\\ \gamma \leq \alpha}}{ \frac{\alpha!}{(\alpha-\gamma)!}} \right) \| u\|_{L^2(S)}^2 \leq (md)^c m^c \| u\|_{L^2(S)}^2 .\qedhere
\]
\end{proof}

\section{Proof of the main theorem}\label{sec:proof-main-theorem}
The aim of this section is to give the proof of Theorem \ref{thm:main-thm-non-radial} assuming the conclusion of Theorem \ref{thm:main-thm-radial-poincare}. Throughout \S \ref{sec:proof-main-theorem}, we assume that $d \geq 5$; the generalization to dimensions $d=2,3,4$ will be given in \S \ref{sec:small-dimensions} and requires an additional input.

At some points of the proof, it will be convenient to work with  an orthonormal basis $\mathcal{B}_m \subset \mathcal{H}_m(\R^d)$, so let us choose one such basis for each $m \geq 0$. Recall that $Z_m^d(x,y) = \sum_{u \in \mathcal{B}_m}{u(x) \overline{u(y)}}$ for all $(x,y) \in  \R^d \times \R^d$ and all $m \in \N_0$ and note that $Z_0^d(x,y) =1$.

Let us start by applying Corollary \ref{cor:cheap-non-radial} with $r_n = \rho_n = \sqrt{n}$ and $c_{p,n}(r) =b_{p,n}(r)$ and $\tilde{c}_{p,n}(r) = {\tilde{b}}_{p,n}(r)$, the numbers provided by Theorem \ref{thm:main-thm-radial-poincare}. In formula \eqref{eq:pre-interpolation-formula} we formally interchange the $n$-sum with the $m$-sum and then the $m$-sum with the integral and are thus motivated to define, for each $(x, \zeta) \in \R^d \times S$ and every $n \geq 1$, the (formal) series
\begin{align}
A_n(x, \zeta) &= \sum_{m=0}^{\infty}{b_{d+2m,n}(|x|) Z_m^d(x, \zeta/ \sqrt{n})}, \label{eq:def-An}\\
\tilde{A}_n(x, \zeta) &= \sum_{m=0}^{\infty}{i^m \tilde{b}_{d+2m,n}(|x|) Z_m^d(x, \zeta/ \sqrt{n})}. \label{eq:def-tilde-An}  
\end{align}
We will address convergence of these series in a moment, but let us observe right away that they trivially converge when $x =0$, with values $A_n(0, \zeta) = b_{d,n}(0)$ and $\tilde{A}_n(0, \zeta) = {\tilde{b}}_{d,n}(0)$. It follows from Corollary \ref{cor:cheap-non-radial} that the formula \eqref{eq:ultimate-interpolation-formula-in-thm} in Theorem \ref{thm:main-thm-non-radial} holds at $x = 0$, because in \eqref{eq:pre-interpolation-formula}, the outer $m$-sum then reduces to the term with $m =0$. The convergence is also absolute in this case, by Theorem \ref{thm:main-thm-radial-poincare}. 

To quantify convergence more generally and more precisely we introduce the following notations.  For each tuple of parameters
 \begin{equation}\label{eq:tuple}
T = (n,\alpha, \beta, \delta, R, s)  \in \N \times \N_0^d \times \N_0^d \times [0, \infty) \times [0, \infty) \times (0,1],
\end{equation}
satisfying $\delta \leq R$ and  for each $m \in \N_0$, we define
\[
S_m(T) = \sup_{\substack{\delta \leq |x| \leq R\\ s\leq |y| \leq s^{-1}}}{\left|\partial_x^{\alpha} \partial_{y}^{\beta} b_{d+2m,n}(|x|) Z_m^d(x,y)n^{-m/2}\right|}  
\]
and $\tilde{S}_m(T)$ analogously by replacing $b_{d+2m,n}$ by $\tilde{b}_{d+2m,n}$. We moreover define
\[
\mathcal{A}(T) = \sum_{m=0}^{\infty}{S_m(T)}, \qquad \tilde{\mathcal{A}}(T) = \sum_{m=0}^{\infty}{\tilde{S}_m(T)}.
\]
The main technical estimates we require are contained in the following lemma.
\begin{lemma}\label{lem:A-inf-summands-bounds}
Fix multi-indices $\alpha, \beta \in \N_0^d$.
\begin{enumerate}[(i)]
\item  For every $s \in (0,1]$, $R>0$ and $n \in \N$, the tuple $T=(n, \alpha,\beta, 0, R,s)$  satisfies $\mathcal{A}(T) < \infty$ and $\mathcal{\tilde{A}}(T) < \infty$. Note here that $\delta = 0$.
\item For all $0 < \delta < R < \infty$, there exists a constant $C > 0$, depending on $\delta, \alpha, R$ and $d$, such that for every $n \in \N$, the tuple $T = (n, \alpha, 0,\delta, R,1)$ satisfies
\[
\max{((\mathcal{A}(T), \tilde{\mathcal{A}}(T))} \leq Cn^{ \frac{5d}{4} + \frac{1}{8}+|\alpha|}.
\]
\end{enumerate}
\end{lemma}

In the arguments below, we will only use Lemma \ref{lem:A-inf-summands-bounds} in the case $\alpha = \beta = 0$. It may be helpful to focus on this special case in a first reading, to avoid distracting details that come from partial derivatives. The statements for general $\alpha, \beta$ imply that the partial sums on the right hand side of formula \eqref{eq:ultimate-interpolation-formula-in-thm} converge uniformly, together with all partial derivatives, on compact subsets of $\R^{d} \setminus \{0 \}$.

\begin{proof}[Proof of Lemma \ref{lem:A-inf-summands-bounds}]
To be able to refer to them later, let us first record the following computations, which follow directly from the generalized Leibniz rule and the formula \eqref{eq:zonal-spherical-in-terms-of-basis}:
\begin{align}
\partial_x^{\alpha} \partial_{y}^{\beta} b_{d+2m,n}(|x|) Z_m^d(x,y) &= \sum_{\gamma_1  + \gamma_2 = \alpha}{ \frac{\alpha!}{\gamma_1! \gamma_2!} \partial_x^{\gamma_1}b_{d+2m,n}(|x|) \partial_x^{\gamma_2} \partial_{y}^{\beta} Z_m^d(x,y)} \label{eq:partial-derivatives-for-the-sup-ST}\\
&= \sum_{u \in \mathcal{B}_m}{   \partial_y^{\beta} \overline{u(y)} \sum_{\gamma_1 + \gamma_2 = \alpha}{  \frac{\alpha!}{\gamma_1! \gamma_2!} \partial^{\gamma_1}_xb_{d+2m,n}(|x|) \partial_x^{\gamma_2}{u(x)}    } }.\label{eq:partial-derivatives-for-the-sup-ST-with-basis}
\end{align}
Whenever an estimate below involves the $\gamma_2$th or $\beta$th derivative of a harmonic polynomial of degree $m$, we may assume that  $|\gamma_2| \leq m$ or $|\beta| \leq m$, as otherwise the derivative vanishes. Moreover, we focus on the estimates for $\mathcal{A}(T)$, which will equally hold for $\tilde{\mathcal{A}}(T)$, because Theorem \ref{thm:main-thm-radial-poincare} gives the same upper bounds for $b_{p,n}$ and $\tilde{b}_{p,n}$.

Part (i) follows basically from the presence of the term $(47/p)^{p/4}$ in the bounds of Theorem \ref{thm:main-thm-radial-poincare} and from Lemma \ref{lem:derivative-harmonic-polynomial}. Turning to details, let $s \in (0, 1]$, $R > 0$ and $n \in \N$  be given. We bound the absolute value of the sum \eqref{eq:partial-derivatives-for-the-sup-ST-with-basis}, for $|x| \leq R$ and $s \leq |y| \leq s^{-1}$, by combining the following estimates:
\begin{itemize}
\item From \eqref{eq:sup-norm-dimension-bound}, $\|u\|_{L^2(S)} = 1$ and Lemma  \ref{lem:derivative-harmonic-polynomial}  we obtain
\begin{align*}
|\partial^{\beta} \overline{u(y)} | &\leq  |y|^{m-|\beta|} \sup_{S^{d-1}}{ |\partial^{\beta} u| } 
\leq  |y|^{m-|\beta|} (\dim{\mathcal{H}_{m-|\beta|}(\R^d)})^{1/2}\|\partial^{\beta}u\|_{L^2(S)} \\
&\ll_{d,|\beta|} s^{|\beta|- m} (m-|\beta|)^{\frac{d-2}{2}} m^{|\beta|}.
\end{align*}
\item Similarly, we find
$
\sup_{|x| \leq R}{|\partial_x^{\gamma_2}{u(x)} |} \ll_{d, \gamma_2} R^m (m-|\gamma_2|)^{\frac{d-2}{2}} m^{|\gamma_2|}
$, for each $\gamma_2 \leq \alpha$.
\item The bound \eqref{eq:bound-partial-b_p-with-zero-in-thm} in Theorem \ref{thm:main-thm-radial-poincare} implies 
\begin{equation}\label{eq:upper-bound-partial-bpn-gamma1-derivative}
\sup_{|x| \leq R}{|\partial_x^{\gamma_1}b_{d+2m,n}(|x|)|} \ll_{d,\gamma_1, R} n^{\frac{d+2m}{2}+|\gamma_1|} \left( \frac{47}{d+2m} \right)^{d/4+m/2}.
\end{equation}
\item The number of terms is $|\mathcal{B}_m| = \dim{\mathcal{H}_m(\R^d)} \ll_d m^{d-2}$, which follows from \eqref{eq:dimension-formula-harmonic-poly}.
\end{itemize}
We deduce that there are $U,X,Y >0$, all depending at most on $d,\alpha, \beta, R,s$ and $n$, so that $
S_m(T) \leq U m^X Y^m (2m+d)^{-m/2}$ for all $m \in \N_0$.  By the root-test or the ratio-test, the series  in part (i) therefore converge, as claimed.

In the remaining part (ii), we will track the dependence on $n$ more precisely. Let $0 < \delta <  R < \infty$ and set $T = (n,\alpha, 0, \delta, R,1)$. We may and will assume that $ \delta < 1 \leq R$. Let $M \geq 1$ be an integral parameter, to be chosen later. We define start and tail sums 
\[
\mathcal{A}_{\text{start}}(T) = \sum_{m=0}^{M}{S_m(T)}, \qquad \mathcal{A}_{\text{tail}}(T) = \sum_{m=M+1}^{\infty}{S_m(T)}.
\]
We start with the analysis of the tail, which is similar to part (i) and we will not yet use that $|x| \geq \delta$. As in the proof of part (i), we use Lemma \ref{lem:derivative-harmonic-polynomial} to bound the derivatives with respect to $x$ of $Z_m^d(x,y)$ appearing in \eqref{eq:partial-derivatives-for-the-sup-ST} by 
\begin{align}
|\partial_x^{\gamma_2} Z_m^d(x, \zeta)| &\ll_{d,|\gamma_2|} |x|^{m-|\gamma_2|} (m-|\gamma_2|)^{\frac{d-2}{2}} m^{|\gamma_2|} \| Z_m^d( \cdot, \zeta)\|_{L^2(S)} \notag \\
&\ll_{d,|\gamma_2|} |x|^{m-|\gamma_2|} m^{d-2+|\gamma_2|}, \label{eq:partial-zonal-sup-bound}
\end{align}
where we used that $\| Z_m^d( \cdot, \zeta)\|_{L^2(S)}^2 = \dim{\mathcal{H}_m(\R^d)}$ and where the implied constants depend neither on $x$, nor on $\zeta$. We have $|x|^{m-|\gamma_2|} \leq R^m$ in \eqref{eq:partial-zonal-sup-bound} and combined with \eqref{eq:upper-bound-partial-bpn-gamma1-derivative}  we see that 
\begin{align*}
 \mathcal{A}_{\text{tail}}(T)&\ll_{d,R, \alpha}  \sum_{m =M+1}^{\infty}{ n^{-m/2}  \left(\frac{47 }{d+2m} \right)^{d/4+m/2} n^{\frac{d+2m}{2}} R^{m} m^{d-2}  \sum_{\gamma_1 + \gamma_2 =\alpha}{  \frac{\alpha!}{ \gamma_1! \gamma_2!} n^{|\gamma_1|} m^{|\gamma_2|}}  }\\
&\ll_{d,R, \alpha}{ n^{d/2+|\alpha|} } \sum_{m =M+1}^{\infty}{ \left( \frac{47 R^2 n}{d+2m} \right)^{m/2}  m^{d-2}(1+m)^{|\alpha|} },
\end{align*}
where we absorbed the term $(47/(d+2m))^{d/4} \ll 1$  into the implied constant and used that the inner sum over $\gamma_1, \gamma_2$  is equal to
\[
(n+m)^{|\alpha|} = (n(1+m/n))^{|\alpha|} \leq n^{|\alpha|}(1+m)^{|\alpha|}.
\]
We now take $M=\lfloor 47 R^2 n \rfloor+2$. Then $\tfrac{ 47 R^2 n}{d+2m} \leq \tfrac{1}{2}$ for all $m \geq M+1$ and hence 
\[
\mathcal{A}_{\text{tail}}(T) \ll_{d, R, \alpha} n^{d/2+|\alpha|} \sum_{m=1}^{\infty}{2^{-m/2} m^{d-2} (1+m)^{|\alpha|} } \ll_{d, \alpha, R} n^{d/2+|\alpha|}.
\]
It remains to bound the finite sum $\mathcal{A}_{\text{start}}(T)$. At this point, the restriction $|x| \geq \delta >0$ becomes important.  By \eqref{eq:bound-partial-b_p-away-from-zero-in-thm} in Theorem \ref{thm:main-thm-radial-poincare} (applied by setting $\varepsilon =1/8$ in its statement) we have, for $\delta \leq |x| \leq R$,
\begin{equation}\label{eq:right-before-amazing-cancellation-happens}
|\partial^{\gamma_1}{b_{d+2m,n}(|x|)}| \ll_{\gamma_1,  R} n^{9/8 + d/4 + m/2+|\gamma_1|}|x|^{-d/2-m+9/4}.
\end{equation}
Crucially, the term $n^{m/2}$ in  \eqref{eq:right-before-amazing-cancellation-happens} cancels with the term $n^{-m/2}$ in the definition of $S_m(T)$ and the term $|x|^{-m}$ in \eqref{eq:right-before-amazing-cancellation-happens} cancels with  $|x|^m$ in \eqref{eq:partial-zonal-sup-bound}. This implies
\begin{align}
\mathcal{A}_{\text{start}}(T) &\ll_{d,R, \alpha} \sum_{m=0}^{M}{ \sup_{ \delta \leq |x| \leq R}{  \sum_{\gamma_1 + \gamma_2 = \alpha}{ \frac{\alpha!}{\gamma_1 ! \gamma_2 !} n^{9/8+d/4+ |\gamma_1|} |x|^{-d/2+9/4} |x|^{-|\gamma_2|} m^{d-2+|\gamma_2| }  } } } \notag\\
& \leq \left(\sup_{\delta \leq |x| \leq R}{ |x|^{-d/2+9/4}  } \right) n^{d/4+9/8} \sum_{m=0}^{M}{   (n+m/\delta)^{|\alpha|} m^{d-2}}. \label{eq:intermediate-bound-for-rmk-convergence-dgeq5}
\end{align}
For $m \leq M$ we can bound
\[
(n+m/\delta)^{|\alpha|} = n^{|\alpha|} \delta^{-|\alpha|}( \delta + \tfrac{m}{n})^{|\alpha|} \leq n^{|\alpha|} \delta^{-|\alpha|}(1 + \tfrac{ 47 R^2 n + 2}{n})^{|\alpha|}  \ll_{R, \alpha} n^{|\alpha|}.
\]
Inserting this into \eqref{eq:intermediate-bound-for-rmk-convergence-dgeq5}, we get
\[
\mathcal{A}_{\text{start}}(T) \ll_{d,R, \delta, \alpha}  n^{d/4+9/8+|\alpha|} (M+1)M^{d-2} \ll_{R,d}  n^{d/4+9/8+|\alpha|+  (d-1)  }= n^{5d/4+1/8 + |\alpha|}.
\]
Thus $\mathcal{A}_{\text{start}}(T)$ dominates $\mathcal{A}_{\text{tail}}(T)$ and this proves part (ii).
\end{proof}
As already mentioned, part (i) of Lemma \ref{lem:A-inf-summands-bounds} implies that for every  $n \in \N$, the series $A_{n}(x, \zeta)$ and $\tilde{A}_{n}(x, \zeta)$ define smooth functions of $(x, \zeta) \in \R^d \times (\R^{d} \setminus \{0 \})$, so they are smooth on $\R^d \times S^{d-1}$. Moreover, it shows that for every continuous function $g : \R^d \rightarrow \C$, the integral $\int_{S}{A_n(x, \zeta) g( \sqrt{n}\zeta) d\zeta}$ defines a smooth function of $x \in \R^d$ such that, for all $\alpha \in \N_0^d$ and $0 < \delta  \leq 1 \leq  R$,
\begin{equation}\label{eq:sup-norm-on-annulus-of-A_n}
\sup_{\delta \leq |x| \leq R}{ \left| \partial_x^{\alpha}\int_{S}{A_n(x, \zeta)}g(\sqrt{n}\zeta) d\zeta \right|} \ll_{d, \delta,R, \alpha} n^{\frac{5d}{4} + \frac{1}{8}+ |\alpha|} \sup_{\zeta \in S}{|g(\sqrt{n}\zeta)|}
\end{equation}
and such that
\begin{equation}\label{eq:interchange-of-m-sum-with-integral}
\int_{S}{A_n(x, \zeta) g( \sqrt{n}\zeta) d\zeta} = \sum_{m=0}^{\infty}{ \int_{S}{b_{d+2m,n}(|x|) Z_m^d(x, \zeta/ \sqrt{n}) g( \sqrt{n}\zeta) d \zeta}}.
\end{equation}
The upper bound \eqref{eq:sup-norm-on-annulus-of-A_n} and the identity \eqref{eq:interchange-of-m-sum-with-integral} also hold for $A_n$ replaced by $\tilde{A}_n$ and $ b_{d+2m,n}$ replaced by $i^m \tilde{b}_{d+2m,n}$.

With these preliminaries in place, we are now ready to prove Theorem \ref{thm:main-thm-non-radial}. Consider any Schwartz function $f: \R^d \rightarrow \C$ and fix a point $x \in \R^d \setminus \{0 \}$. The sequences of the suprema of $f$ and $\hat{f}$ over the spheres of radius $\sqrt{n}$ then decay rapidly.  Together with part (ii) of Lemma \ref{lem:A-inf-summands-bounds}, applied with $T =(n,0,0,|x|,|x|,1)$, it follows that the double series
\begin{equation}\label{eq:interchange-of-n-sum-with-m-sum}
\sum_{n=1}^{\infty}\int_{S}{A_{n}(x, \zeta) f(\sqrt{n}\zeta) d\zeta} = \sum_{n=1}^{\infty}{\sum_{m=0}^{\infty}{b_{d+2m,n}(|x|) \int_{S}{ Z_m^d(x, \zeta/ \sqrt{n}) f(\sqrt{n}\zeta) d\zeta}}},
\end{equation}
converges absolutely, as does the one involving $\hat{f}$, $\tilde{A}_{n}$ and $\tilde{b}_{d+2m,n}$. By Fubini--Tonelli on $\N \times \N_0$, we can therefore interchange the sum over $n$ with that over $m$. Then, combining \eqref{eq:interchange-of-n-sum-with-m-sum} with \eqref{eq:interchange-of-m-sum-with-integral} and Corollary \ref{cor:cheap-non-radial}, we deduce that the left hand side of \eqref{eq:interchange-of-n-sum-with-m-sum}, plus the corresponding series involving $\tilde{A}_n$ and $\hat{f}$, equals $f(x)$. This proves our interpolation formula \eqref{eq:ultimate-interpolation-formula-in-thm} in Theorem \ref{thm:main-thm-non-radial} for the point $x \neq 0$. Finally, recall that we already proved it for $x = 0$, right after the definition of $A_n(x,\zeta)$, $\tilde{A}_n(x,\zeta)$. This completes the proof of Theorem \ref{thm:main-thm-non-radial}, up to the proof of Theorem \ref{thm:main-thm-radial-poincare}, which will be given in \S \ref{sec:poincare-type-series}.

\subsection{Remarks on (uniform) convergence}\label{sec:rmks-convergence}
If we keep track of the implied constants in the proof of part (ii) in Lemma \ref{lem:A-inf-summands-bounds} in the case $|\alpha| = |\beta|=0$, we obtain the following explicit bound. For any $0 < \delta \leq 1 \leq R$ and every $n \in \N$, the supremum $\sup_{\delta \leq |x| \leq R, |\zeta|=1}{|A_n(x, \zeta)|}$ is less than or equal to
\begin{equation}
 C_2 H_d (1/ \delta)^{d/2-9/4}n^{d/4+9/8} (47 n R^2 +3)^{d-1}  + C_1 H_d(47/d)^{d/4} \sum_{m=1}^{\infty}{2^{-m/2}m^{d-2}}, \label{eq:explicit-constant} 
\end{equation} 
where $H_d =  \tfrac{2}{(d-2)!} \sup_{m \in \N_0}{ \frac{\dim{\mathcal{H}_m(\R^d)}}{m^{d-2}}   }$, compare with \eqref{eq:dimension-formula-harmonic-poly} and where $C_1, C_2 > 0$ are constants as in part (i) of Theorem \ref{thm:main-thm-radial-poincare}. We deduce that the interpolation formula \eqref{eq:ultimate-interpolation-formula-in-thm} converges uniformly and rapidly on every $d$-dimensional annulus, equivalently on any compact subset avoiding the origin. Note moreover that, if $2 \leq d \leq 4$,  then $(1/ \delta)^{d/2-9/4} \leq 1$ and the proof shows that we have uniform convergence on any compact subset of $\R^d$.
\subsection{Reformulation of the proof}\label{sec:alternative-proof}
We can formulate the above proof of Theorem \ref{thm:main-thm-non-radial}  in a way that is more reminiscent of \cite{R-V} (or \cite{CKMRV-E8-leech}). Namely, we can fix a vector $x \in \R^d$ and interpret the right hand side of the interpolation formula \eqref{eq:ultimate-interpolation-formula-in-thm} as a linear functional $\ell_{x} : \mathcal{S}(\R^d) \rightarrow \C$. Note that it is indeed \emph{defined} on all of $\mathcal{S}(\R^d)$ by Lemma \ref{lem:A-inf-summands-bounds} and moreover continuous. It therefore suffices to show that $\ell_x(f) = f(x)$, for $f$ in a generating set of a dense subspace of $\mathcal{S}(\R^d)$. Arguing as in the proof of Proposition \ref{prop:lifts}, we can therefore reduce to $f(x) = u_0(x) e^{\pi i z_0 |x|^2}$, where $u_0 \in \mathcal{B}_{m_0}$ and $z_0 \in \H$ are fixed. In this case, the desired identity $\ell(x) =f(x)$ reduces to the formula \eqref{eq:interpolation-formula-radial-in-main-thm-poincare} in Theorem \ref{thm:main-thm-radial-poincare}, in dimension $p = d+2m_0$, applied to the Gaussian.

\section{Dimensions 2, 3 and 4}\label{sec:small-dimensions}
To extend Theorem \ref{thm:main-thm-non-radial} to dimensions $2,3$ and $4$, we need the following input. 
\begin{proposition}\label{prop:radial-input-small-dimensions}
For every $p \in \{2,3,4\}$, there exist sequences $(a_{p,n})_{n \in \N_0}$, $(\tilde{a}_{p,n})_{n \in \N_0}$ of radial Schwartz functions  on $\R^p$ such that, for every $f \in \mathcal{S}(\R^p)$ and every $x \in \R^p$, 
\begin{equation}\label{eq:radial-input-small-dimensions}
f(x) = \sum_{n=0}^{\infty}{a_{p,n}(x) f(\sqrt{n})} +\sum_{n=0}^{\infty}{\tilde{a}_{p,n}(x) \hat{f}(\sqrt{n}) },
\end{equation}
where the series converge absolutely and such that, for every continuous semi-norm $\|\cdot\|$ on $\mathcal{S}(\R^p)$, the sequences $(\|a_{p,n}\|)_{n \in \N_0}$ , $(\|\tilde{a}_{p,n}\|)_{n \in \N_0}$ are of polynomial growth.
\end{proposition}
\begin{proof}
This follows from more general results by Bondarenko, Radchenko and Seip \cite{BRS}. In the notation of their paper, we specialize the discussions in section 3 of \cite{BRS} to the function $\varphi(z) = e^{\pi i z r^2}$, where $r =|x| \in \R_{\geq 0}$ and to the parameter $k = p/2$ (their results would in fact cover all real $k \geq 0$). The Fourier coefficients of the series denoted  $F_{k}^{\pm}(\tau, \varphi)$ give the Fourier- even and -odd parts of the radial functions $a_{p,n}$ and $\tilde{a}_{p,n}$ is the Fourier transform of $a_{p,n}$ on $\R^{p}$. The interpolation formula \eqref{eq:radial-input-small-dimensions} follows from the density of complex Gaussians (Lemma \ref{lem:density-gaussians-and-continuity}) together with the functional equations satisfied by the generating series $F_{k}^{\pm}(\tau, \varphi)$, as in \cite{R-V} (see also \S \ref{sec:generating-series-and-functional-equations} for a related discussion). The same technique as in \cite{R-V} can be used to prove that the functions $a_{p,n}, \tilde{a}_{p,n}$ belong to the Schwartz space and that all their Schwartz semi-norms grow polynomially with $n$ (see also Proposition 6.1 in \cite{BRS}).
\end{proof}

\begin{theorem}\label{thm:non-radial-small-dimensions}
Let $d  \in \{2, 3,4 \}$. For every $n \geq 1$, there are two smooth functions $A_n, \tilde{A}_n : \R^d \times S^{d-1} \rightarrow \C$ and for every multi-index $\alpha \in \N_0^d$ of size $|\alpha| \leq 1$, two Schwartz functions $h_{\alpha}, \tilde{h}_{\alpha} \in \mathcal{S}(\R^d)$ such that, defining
\[
T_x(f) = \sum_{|\alpha| \leq 1}{h_{\alpha}(x) (\partial^{\alpha}f)(0)}, \quad \tilde{T}_x(g) = \sum_{|\alpha| \leq 1}{{\tilde{h}}_{\alpha}(x) (\partial^{\alpha} g)(0)},
\]
for $f,g \in \mathcal{S}(\R^d)$ and $x \in \R^d$, the following holds. For all $f \in \mathcal{S}(\R^d)$ and all $x \in \R^d$,
\begin{equation}\label{eq:ultimate-interpolation-formula-in-last-section}
f(x) = T_x(f) +\sum_{n=1}^{\infty}{\int_{S}{A_n(x, \zeta) f(\sqrt{n}\zeta) d\zeta}} +  \tilde{T}_x(\hat{f}) + \sum_{n=1}^{\infty}{\int_{S}{{\tilde{A}}_n(x, \zeta) \hat{f}(\sqrt{n}\zeta) d\zeta}}
\end{equation}
 and both series converge absolutely.
\end{theorem}
\begin{proof}
We modify the arguments in \S \ref{sec:proof-main-theorem} as follows. First, we define the integers $M_2 = 2$, $M_3 =1$, $M_4 = 1$. We start with Corollary \ref{cor:cheap-non-radial} and apply it with inputs $r_n = \rho_n = \sqrt{n}$ and $c_{p,n}(r)$ and $\tilde{c}_{p,n}(r)$ taken as follows, depending on the dimension $d$ of interest:
\begin{align*}
&(c_{d+2m,n}(r), \tilde{c}_{d+2m,n}(r))  = (a_{d+2m,n}(r),\tilde{a}_{d+2m,n}(r)) &&\text{ if }\, m < M_d,\\
&(c_{d+2m,n}(r), \tilde{c}_{d+2m,n}(r))  = (b_{d+2m,n}(r),\tilde{b}_{d+2m,n}(r)) &&\text{ if }\, m \geq M_d,
\end{align*}
where $b_{d+2m,n}$ and $\tilde{b}_{d+2m,n}$ are as in Theorem \ref{thm:main-thm-radial-poincare} and $a_{d+2m,n}, \tilde{a}_{d+2m,n}$ are as in Proposition \ref{prop:radial-input-small-dimensions} (and we abuse notation as in \S \ref{sec:notation-radial-near-origin}). We then redefine the series $A_n$ in \eqref{eq:def-An}  to
\[
A_n(x, \zeta) = \sum_{m=0}^{\infty}{c_{d+2m,n}(|x|) Z_m^d(x, \zeta/ \sqrt{n})}
\]
and redefine $\tilde{A}_n$ in \eqref{eq:def-tilde-An} in the same way, replacing $\tilde{b}_{d+2m}$ by $\tilde{c}_{d+2m}$. Again, these series trivially converge at $x = 0$ and the formula \eqref{eq:ultimate-interpolation-formula-in-last-section} holds in this case  by Corollary \ref{cor:cheap-non-radial}. Notice that they differ by at most two terms from the ones that involved only $b_{d+2m,n}$, $\tilde{b}_{d+2m,n}$. By the assumption on the semi-norms of $a_{p,n}, \tilde{a}_{p,n}$, we can control the ``exceptional" terms by
\[
|a_{d+2m,n}(|x|)Z_m^d(x, \zeta)| \ll_d \left(\sup_{\xi \in \R^d}{a_{d+2m,n}(|\xi|) |\xi|^m} \right) m^{d-2} \ll n^{B}m^{d-2},
\]
where $B > 0$ depends only on $d$ (because at most two values of $m$ need to be considered here). It follows that the new functions $A_n$, $\tilde{A}_n$ obey bounds similar to those stated in Lemma \ref{lem:A-inf-summands-bounds}. The functions $h_{\alpha}$, $\tilde{h}_{\alpha}$ arise from Corollary \ref{cor:cheap-non-radial} as follows. In the double sum \eqref{eq:pre-interpolation-formula}, we split the inner $n$-sum into the sub-sums over $n \in \{0 \}$ and $n \in \N$ and then interchange (as we may) the outer sum with these inner sums individually. Doing so, we see that
\[
h_{\alpha}(x) = \frac{1}{\alpha!} \int_{S}{\zeta^{\alpha} Z_m^d(x, \zeta) d\zeta}\, a_{d+2m,0}(|x|)=  \frac{1}{\alpha!}\sum_{u \in \mathcal{B}_m}{ \left(\int_{S}{\zeta^{\alpha} \overline{u(\zeta)} d\zeta} \right)  a_{d+2m,0}(|x|) u(x)},
\]   
where $S = S^{d-1}$ and $\mathcal{B}_m \subset \mathcal{H}_m(\R^d)$ is an orthonormal basis. In this way we  can prove \eqref{eq:ultimate-interpolation-formula-in-last-section}, with point-wise absolute convergence (but recall also the remarks regarding uniform convergence made at the end of \S \ref{sec:rmks-convergence}).
\end{proof}

\section{Poincar{\'e} series-type construction}\label{sec:poincare-type-series}
The goal of \S \ref{sec:poincare-type-series} is to prove Theorem \ref{thm:main-thm-radial-poincare}. Basic preliminaries on modular forms follow in \S \ref{sec:modular-preliminaries} and the general proof strategy via generating series and functional equations, following \cite{R-V, CKMRV-E8-leech}, is explained in \S \ref{sec:generating-series-and-functional-equations}. After some group theoretic preliminaries in \S \ref{sec:a-particular-set-of-words}, the definition of the solutions to the above mentioned functional equations, as well as the definition of the functions $b_{p,n}, \tilde{b}_{p,n}$ in Theorem \ref{thm:main-thm-radial-poincare}, is given in  \S \ref{sec:definition-of-generating-series}. The required growth estimates are then proved in \S \ref{sec:growth-estimates-poincare}.
\subsection{Modular preliminaries}\label{sec:modular-preliminaries}
We assemble some basic facts related to modular forms that are relevant for our purposes. As general references, we mention \cite{Iwaniec, SerreCourse, MumfordTata, rankin-book}.
\subsubsection{Fractional linear transformations}\label{sec:fractional-linear-transformations}
We let $\SL_2(\R)$ and its subgroups act on the upper half plane $\H$ by fractional linear transformations. For $M = \begin{psmallmatrix} a & b \\ c &d \end{psmallmatrix} \in \SL_2(\R)$ and $\tau \in \H$ we define $j(M, \tau) = c \tau + d$ and we recall that $\imag(M \tau) = \imag(\tau) |j(M,\tau)|^{-2}$. For $M \in \SL_2(\R)$ we use $[M]$ to denote its image in $\PSL_2(\R)$ and similarly for elements of subgroups $\Gamma \leq \SL_2(\R)$ containing $-I$. We write $\overline{\Gamma}$ for the image of such a subgroup in $\PSL_2(\R)$. 
\subsubsection{Congruence subgroups of level 2}\label{sec:congruence-subgroups}
We use $S = \begin{psmallmatrix} 0 & -1\\ 1 & 0 \end{psmallmatrix}$ and $T = \begin{psmallmatrix} 1 & 1\\ 0 & 1 \end{psmallmatrix} \in \SL_2(\Z)$, which together generate the group $\SL_2(\Z)$. Let $\text{pr}_2: \SL_2(\Z) \rightarrow \SL_2(\Z/ 2 \Z)$ denote the natural morphism. The principal congruence subgroup of level $2$ is the normal subgroup $\Gamma(2) = \ker{(\text{pr}_2)} \triangleleft \SL_2(\Z)$. It is generated by $-I,T^2, ST^2 S$. The group $\overline{\Gamma(2)}$ is \emph{freely} generated by $[T^2]$ and $[S T^2 S]$. The theta subgroup is $\Gamma_{\theta}= \text{pr}_2^{-1}(\{1, \text{pr}_2(S)\})$ and equal to $ \Gamma(2) \sqcup S \Gamma(2)$ and moreover generated by $S$ and $T^2$.
\subsubsection{Jacobi's theta function}\label{sec:jacobi-theta}
For $(z, \tau) \in \C \times \H$, let $\vartheta(z, \tau) = \sum_{n \in \Z}{e^{\pi i n^2 \tau + 2 \pi i n  z}}$ denote Jacobi's theta function and let $\Theta_3(\tau)=\theta_{00}(\tau) =  \vartheta(0, \tau)$ denote one of its Nullwerte, following historical notations. This series converges normally on $\H$ and it is well-known that $\Theta_3$ never vanishes on $\H$, by Jacobi's celebrated triple product formula (for example). We may therefore define, for all $(M, \tau) \in \PSL_2(\R) \times \H$, the number $j_{\Theta}(M,\tau) = \Theta_3(M \tau)/ \Theta_3(\tau) \in \C^{\times}$. The Poisson summation formula for \emph{even} Schwartz functions on $\R$ is equivalent to $j_{\Theta}(S,\tau) = (-i\tau)^{1/2}$ (Lemma \ref{lem:density-gaussians-and-continuity} and \S \ref{sec:notation-square-roots}) and the identity $j_{\Theta}(T^2, \tau) = 1$ is trivial. Since $\Gamma_{\theta}$ is generated by $S$ and $T^2$, it follows that $\Theta_3^8$ transforms like a modular form of weight $4$ on $\Gamma_{\theta}$ and that
\begin{equation}\label{eq:automorphic-factors-the-same}
|j(M,\tau)| = |j_{\Theta}(M,\tau)|^{2} \quad \text{for all} \quad (M,\tau) \in \overline{\Gamma_{\theta}} \times \H.
\end{equation}  
We give more information on the transformation laws of $\Theta_3$ in \S \ref{sec:bessel-kloosterman} and introduce its accompanying theta constants $\Theta_2, \Theta_4$ in \S \ref{sec:space-of-relations-infinite-dimensional}, but these things will not be needed in the remainder of \S \ref{sec:poincare-type-series}.
\subsubsection{Slash action}\label{sec:slash-action}
For any half-integer $k \in \tfrac{1}{2}\Z$ and any complex vector space $\mathcal{S}$ (e.g. $\mathcal{S} = \mathcal{S}_{\text{rad}}(\R^p)$ or $\C$), we define the slash-action in weight $k$ on the space of all functions $f: \H \rightarrow \mathcal{S}$, by $(f|_{k}M)(z) = j_{\Theta}(M,z)^{-2k} f(Mz)$. We extend it linearly to the group ring $\C[\PSL_2(\R)]$.

\subsection{Generating series and functional equations}\label{sec:generating-series-and-functional-equations}
As part of the  proof Theorem \ref{thm:main-thm-radial-poincare}, we explain here the general strategy to prove an interpolation formula for radial Schwartz functions on $\R^d$, by rephrasing the problem in terms of certain holomorphic functions on the complex upper half plane. This strategy is very similar to the one used in \cite{R-V} and also similar to the more complicated one used in \cite{CKMRV-E8-leech}. We shall implement it in \S \ref{sec:definition-of-generating-series} and \S \ref{sec:growth-estimates-poincare}. 

Suppose we want to find radial functions $a_{n}, \tilde{a}_{n} $ on $\R^p$ such that for all $f \in \mathcal{S}_{\text{rad}}(\R^p)$ and all $x \in \R^p$,
\begin{equation}\label{eq:radial-fip-formula-sec-radial-fip}
f(x) = \sum_{n=0}^{\infty}{f(\sqrt{n}) a_{n}(x) } + \sum_{n=0}^{\infty}{\hat{f}(\sqrt{n}) \tilde{a}_{n}(x) }
\end{equation}
with absolute convergence. Fixing a point $x \in \R^p$ we may think of \eqref{eq:radial-fip-formula-sec-radial-fip} as an identity of linear functionals on $\mathcal{S}_{\text{rad}}(\R^p)$. From this point of view, it is reasonable to search among sequences $(a_n(x))_{n \in \N_0}$, $(\tilde{a}_n(x))_{n \in \N_0}$ that grow at most polynomially in $n$, because in this case, the right hand side of \eqref{eq:radial-fip-formula-sec-radial-fip} also defines a \emph{continuous} linear functional and the validity of \eqref{eq:radial-fip-formula-sec-radial-fip} becomes equivalent to the validity of the same equation for $f$ belonging to a (generating set of a) dense subspace of $\mathcal{S}_{\text{rad}}(\R^p)$. Such a set is given by $\{G_p(\tau) \,:\,\tau \in \H\}$, by Lemma \ref{lem:density-gaussians-and-continuity}.
Requiring polynomial growth on the coefficients also implies that the generating series 
\[
F(\tau,x) = \sum_{n=0}^{\infty}{a_n(x) e^{\pi i n \tau}}, \quad \tilde{F}(\tau,x) = \sum_{n=0}^{\infty}{ \tilde{a}_n(x) e^{\pi i n \tau }}
\]
converge absolutely for all $\tau \in \H$ and $x \in \R^p$. If \eqref{eq:radial-fip-formula-sec-radial-fip} holds for all $f$, then in particular for $f = G_p(\tau)$, and hence the following set of functional equations must be satisfied by $F$, $\tilde{F}$. We write these without the variables $x, \tau$ and we use the slash action of $\C[\PSL_2(\Z)]$ in weight $k=p/2$, as defined in \S \ref{sec:slash-action}.
\begin{enumerate}[(i)]
\item $F + \tilde{F}|_{k}S = G_p$.
\item $F|_{k}(T^2-1) = 0$.
\item $\tilde{F}|_{k}(T^2-1) = 0$.
\item $F|_{k}(ST^2 S - 1) = G_p|_{k}(ST^2 S -1)$.
\end{enumerate}
Here, equation (iv) is implied by all the others and equation (iii) is implied by all the others. The formal verification is left to the reader. Conversely, if we can find, in the first place, two functions $F, \tilde{F} : \H \times \R^p \rightarrow \C$  that are holomorphic and 2-periodic in the first variable, radial in the second and moreover related by (i), then we can define $a_n(x)$ as the $n$th Fourier coefficient of $\tau \mapsto F(\tau,x)$ and $\tilde{a}_n$ as the $n$th Fourier coefficient of $\tau \mapsto \tilde{F}(\tau,x)$. To prove \eqref{eq:radial-fip-formula-sec-radial-fip}, it then only remains to be shown that $a_{n} = 0= \tilde{a}_{n}$ for $n < 0$ and that the polynomial growth requirement holds. 
\subsection{A particular set of words}\label{sec:a-particular-set-of-words}
We continue with our preparations for the proof of Theorem \ref{thm:main-thm-radial-poincare}, outlined at the beginning of \S \ref{sec:poincare-type-series}, by introducing and studying a certain subset of $\overline{\Gamma(2)}$, that will enter the definition of the generating series in the next subsection.

As for notation, for an element $M \in \SL_2(\Z)$, we denote by $[M]$ its class modulo $\{\pm I\}$, but we also use $\bar{S}= [S]$ in  this section. Note that $\bar{S}^2 = 1 \in \PSL_2(\Z)$. If $M = \begin{psmallmatrix}
a & b\\ c & d
\end{psmallmatrix}$, then we will often write $a = a_M$, $b = b_M$, $c = c_M$ and $d = d_M$. When it is unambiguous, we use the same notation for $M \in \PSL_2(\Z)$, for example,  writing  $|c_M| \geq 1$ or a ratio of matrix entries. We recall that the group $\overline{\Gamma(2)}$ is freely generated by the elements $A = [T^2]$ and $B = [S T^2 S]$. We also use the representatives $A_{0} = T^2$, $B_0 = ST^2S^{-1}$ in this section. 

\begin{definition}\label{def:set-B}
The subset $\mathcal{B} \subset \overline{\Gamma(2)}$ is defined as the set of all nonempty finite reduced words in $A$ and $B$ that start with a nonzero power of $B$. More formally, an element $M \in \overline{\Gamma(2)}$ belongs to $\mathcal{B}$, if and only if there are integers $m \geq 1$ and $e_1,\dots, e_m, f_1, \dots, f_{m}$, all nonzero, except possibly $e_m$, such that $M = B^{f_1}A^{e_1} \cdots B^{f_m}A^{e_m}$. We define the set $\tilde{\mathcal{B}}= \mathcal{B}\overline{S} \sqcup \{ \overline{S} \}  = \{ M \overline{S} \,:\, M \in \mathcal{B} \} \sqcup \{ \overline{S} \} \subset \overline{\Gamma_{\theta}}$.
\end{definition} 
We shall prove that the elements $\mathcal{B}$ and those of $\tilde{\mathcal{B}}$ are uniquely determined by their bottom rows (up to sign). To formulate this precisely, we define 
\begin{align*}
\mathcal{P} &= \{ (c,d) \in \Z^2 \,:\, \gcd(c,d) = 1, \quad  c \equiv 0, d \equiv 1 \pmod{2},\, c \neq 0 \},\\
\tilde{\mathcal{P}} &= \{ (c,d) \in \Z^2 \,:\, \gcd(c,d) = 1, \quad  c \equiv 1, d \equiv 0 \pmod{2} \}.
\end{align*}
The unit group $\Z^{\times} = \{-1, 1 \}$ acts on these sets in the obvious way, via $\varepsilon \cdot (c, d) = (\varepsilon c, \varepsilon d)$, for $\varepsilon \in \Z^{\times}$. We further equip them with an action of $\Z$, defined as $(c,d)| \ell = (c, d+2\ell c)$. These actions commute, so that $\Z$ acts on the quotients $\mathcal{P} / \Z^{\times}$, $\tilde{\mathcal{P}} / \Z^{\times}$. We write the class of  $(c,d)$ in these quotients as $[(c,d)]= \{(c,d), (-c,-d)\}$.  
\begin{lemma}\label{lem:identifying-sets-B-tilde-B}
With notations as above, the following holds.
\begin{enumerate}[(i)]
\item For each $M \in \mathcal{B}$, $\tilde{M} \in \tilde{\mathcal{B}}$ and each $\ell \in \Z$ one has $M A^{\ell} \in \mathcal{B}$ and $\tilde{M} A^{\ell} \in \tilde{\mathcal{B}}$. In other words, the group $\Z \cong \langle A \rangle$ acts on either set $\mathcal{B}$, $\tilde{\mathcal{B}}$ by right multiplication.
\item The assignment
\begin{equation}\label{eq:bottom-row-assignment}
\left[ \begin{pmatrix}
 a & b \\ c & d
\end{pmatrix} \right] \mapsto [(c,d)]
\end{equation}
defines $\Z$-equivariant bijections 
$
\mathcal{B}  \cong \mathcal{P}/\Z^{\times}$, $\tilde{\mathcal{B}}  \cong \tilde{\mathcal{P}}/\Z^{\times}.
$
\end{enumerate}
\end{lemma}
\begin{proof}
We prove part (i).  Let $M \in \mathcal{B}$, $\tilde{M} \in \tilde{\mathcal{B}}$, $\ell \in \Z$. It follows directly from the definition that $M A^{\ell} \in \mathcal{B}$. As for $\tilde{M}$, write $\tilde{M}= H\bar{S}$ for some $H \in \mathcal{B} \sqcup \{1\}$. Then \[
\tilde{M}A^{\ell} = H\bar{S}A^{\ell} = H\bar{S}A^{\ell} \bar{S} \bar{S} =HB^{\ell} \bar{S}.
\]
and we deduce  $\tilde{M}A^{\ell} \in \tilde{\mathcal{B}}$ in all cases; it equals $\bar{S}$ if $H = B^{- \ell}$ and $H B^{\ell}$ belongs to $\mathcal{B}$ otherwise.


We prove part (ii). In general, the assignment \eqref{eq:bottom-row-assignment} defines a mapping $\PSL_2(\Z) \rightarrow \Z_{\text{prim}}^2/ \Z^{\times}$, where $\Z_{\text{prim}}^2$ denotes the set of all primitive row vectors in $\Z^2$ (nonzero vectors with coprime entries). Also in general, two elements $X_1, X_2 \in \overline{\Gamma_{\theta}}$ have the same image under \eqref{eq:bottom-row-assignment}, if and only if there is $\ell \in \Z$ so that $X_2 = A^{\ell}X_1$, as a short calculation shows. 

We now prove the assertion about $\mathcal{B}$, proving that the map is well-defined, injective and surjective one after the other.

First, it maps indeed to $\mathcal{P}$, because no element of $\mathcal{B}$ can have lower left entry zero. Indeed, elements of $\overline{\Gamma(2)}$ have lower left-entry equal to zero, if and only if they belong to $\langle A \rangle$ and $\langle A \rangle \cap \mathcal{B} = \emptyset$ holds by definition. 

Let $M_1, M_2 \in \mathcal{B}$ and suppose they have the same image under \eqref{eq:bottom-row-assignment}. This implies that $M_2 = A^{\ell}M_1$ for some $\ell \in \Z$. By definition of $\mathcal{B}$ and and the fact that $A$ and $B$ freely generate $\overline{\Gamma(2)}$, this implies that $\ell = 0$, so our map is injective.

It remains to establish surjectivity. Let $(c_0,d_0) \in  \mathcal{P}$ such that $c_0 > 0$. Recall that $c_0$ is even and $d_0$ is odd by definition. Since $\gcd(2c_0, d_0) = 1$ we may choose $a_0, b_0 \in \Z$ such that $
M_0 = \begin{psmallmatrix}
 a_0 & b_0\\ c_0 & d_0
\end{psmallmatrix} \in \Gamma(2).
$
It then suffices to find $h \in \Z$ so that $A^{h} [M_0] \in \mathcal{B}$, because this element will still map to $[(c_0, d_0)]$. One may find such an $h$, via repeated reduction of the bottom entries mod $2d_0$ and $2c_0$, implemented via the formulas\footnote{Since the bottom row entries of matrices in $\Gamma(2)$ are of opposite parities, at least one of them reduces by at least 1 in absolute value, in each step in the successive reductions described above. If, say $M_0 B^{\ell_1} A^{m_1}B^{\ell_2}$ has lower left entry zero, this product equals $A^{-h}$, for some $h \in \Z$ and hence $A^{h}M_0 = B^{-\ell_2}A^{-m_1}B^{-\ell_1}$. Now $\ell_2 \neq 0$, as otherwise the process would have ended earlier, namely when $MB^{\ell_1}$ had lower left entry zero. In fact, we will not need surjectivity in the proof of Theorem \ref{thm:main-thm-radial-poincare}. It will only be used in the supplementary section \ref{sec:bessel-kloosterman}. }
\begin{align*}
 \begin{pmatrix}
a & b\\ c & d
\end{pmatrix} A_0^{m} &=\begin{pmatrix}
a & b\\ c & d
\end{pmatrix} \begin{pmatrix}
1 & 2m\\ 0 & 1
\end{pmatrix} = \begin{pmatrix}
a & b + 2am\\
c & d + 2cm
\end{pmatrix},\\
\begin{pmatrix}
a & b\\ c & d
\end{pmatrix} B_0^{\ell} &=\begin{pmatrix}
a & b\\ c & d
\end{pmatrix} \begin{pmatrix}
1 & 0\\ -2\ell & 1
\end{pmatrix} = \begin{pmatrix}
a-2b\ell & b\\
c-2d\ell & d
\end{pmatrix}.
\end{align*}
We will now deduce that the map \eqref{eq:bottom-row-assignment} also induces a bijection $\tilde{\mathcal{B}} \cong \tilde{\mathcal{P}}/ \Z^{\times}$. It is well-defined because
\begin{equation}\label{eq:right-multiply-by-S}
\begin{pmatrix}
a & b\\ c & d
\end{pmatrix} \begin{pmatrix}
0 & -1\\
1 & 0
\end{pmatrix} = \begin{pmatrix}
b & -a\\ d & -c
\end{pmatrix}.
\end{equation}
It is injective, because if $\tilde{M}_1 = M_1\bar{S}$, $\tilde{M}_2 = M_2 \bar{S}$, $M_i \in \mathcal{B} \sqcup \{ 1 \}$ map to the same element of $\tilde{\mathcal{P}}$, then, by the above general remark on the assignment \eqref{eq:bottom-row-assignment}, we have $\tilde{M}_2 = A^{\ell} \tilde{M_2}$, for some $\ell \in \Z$, equivalently $M_2 = A^{\ell}M_1$, hence $\ell = 0$. Finally, to show surjectivity, let $(c,d) \in \tilde{\mathcal{P}}$. By definition, $c, d$ are coprime integers, $c$ is odd and $d$ is even. There are two cases:
\begin{itemize}
\item $d = 0$. Then $c  \in \{-1, 1\}$ and $[S]$ maps to $[(c,d)]$ under \eqref{eq:bottom-row-assignment}.
\item $d \neq 0$. Then $[(d,-c)] \in \mathcal{P}$ and by what we have shown above, there is $M \in \mathcal{B}$ mapping to $[(d,-c)]$. By \eqref{eq:right-multiply-by-S}, the element $M\bar{S} \in \tilde{\mathcal{B}}$ then maps to $[(-c,-d)] = [(c,d)]$, as required.
\end{itemize} 
This concludes the proof of Lemma \ref{lem:identifying-sets-B-tilde-B}.
\end{proof}
The next lemma and its corollary will be used for certain estimates in \S \ref{sec:growth-estimates-poincare} in combination with the  useful identity
\begin{equation}\label{eq:surprisingly-useful-identity}
\frac{a\tau + b}{c \tau +d} = \frac{a}{c}- \frac{1}{c(c\tau + d)},
\end{equation}
which holds for all $\tau \in \H$ and all $a,b,c,d \in \R$, satisfying with $c \neq 0$ and $ad-bc  =1$.
\begin{lemma}\label{lem:a-less-than-c}
For every $ M \in \mathcal{B}$ we have $|a_M| \leq |c_M|$ and $|b_M| \leq |d_M|$ and for every $\tilde{M} \in \tilde{\mathcal{B}}$ we have $|a_{\tilde{M}}| \leq |c_{\tilde{M}}|$.
\end{lemma}
\begin{proof}
Since  right-multiplication by $\bar{S}$ interchanges columns  \eqref{eq:right-multiply-by-S} and since the upper left entry of $\bar{S}$  is zero, it suffices to prove the assertion about elements of $\mathcal{B}$.  We do this via induction on the word length of $M \in \mathcal{B}$, but we will add letters \emph{on the left} in the inductive step. We first compute generally, for any $a,b,c,d,m,\ell \in \Z$, that
\[
B_0^{-m}A_0^{\ell} \begin{pmatrix}
 a & b \\ c & d
\end{pmatrix} = \begin{pmatrix}
a + 2\ell c & b + 2 \ell d\\
2 am + c(1+4 m \ell) & 2m b + d(4 m \ell + 1)
\end{pmatrix}.
\]

\emph{Base case:} In the above, take $a =d = 1$, $c = b = 0$ and assume that $m \neq 0$.  We need to show that $|1| \leq |2m|$ and $|2 \ell| \leq |4m \ell +1|$. This is immediate.

\emph{Inductive step:} We assume that $|a| \leq |c|$, $|b| \leq |d|$ and that $m \ell \neq 0$. We need to show:
\begin{enumerate}
\item $|a+ 2 \ell c| \leq |2 am + c(1+ 4 m \ell)|$,
\item $|b+2 \ell d| \leq |2mb + d(1 + 4 m \ell)|$.
\end{enumerate}
If $c = 0$, then (1) holds trivially and  if $d = 0$, then (2) holds trivially (since $m \neq 0$). We therefore assume that $cd\neq 0$. Dividing then (1) by $|c|$ and (2) by $|d|$, the inductive hypothesis reduces our task to showing that for all $q \in [-1,1] \cap \Q$, 
\[
|q+ 2 \ell| \leq |2m q + (1+4 m \ell)| = |2m(q+2\ell) + 1|.
\]
Introduce $y = q+ 2\ell$, so that what we want to show is $|y| \leq |2my +1|$. But indeed,
\[
|2my+1| \geq 2|m||y|-1 \geq 2|y|-1 \geq |y|,
\]
since $|m| \geq 1$ and $|y| = |2 \ell + q|\geq 2- |q| \geq 1$ since $|q| \leq 1$.
\end{proof}
\begin{corollary}\label{cor:a-less-than-c}
Let $\Omega \subset \H$ be a compact set. Then $\sup_{( \tau, M ) \in \Omega \times( \mathcal{B} \cup \tilde{\mathcal{B}})}{|M\tau|} < \infty$.
\end{corollary}
\begin{proof} 
For $z \in \H$ write $\Lambda_{z} = \Z z + \Z \subset \C$ for the lattice generated by $z$ and $1$ and, for any lattice $\Lambda \subset \C$, write $\mathfrak{s}(\Lambda) = \inf_{0 \neq \lambda \in \Lambda}{|\lambda|}$ for the length of its shortest vectors. The assignment $z \mapsto \mathfrak{s}(\Lambda_z)$  defines a continuous function $\H \rightarrow (0, + \infty)$, as is well-known. Now let  $M  \in \mathcal{B} \cup \tilde{\mathcal{B}}$ be represented by $\begin{psmallmatrix} a &b \\ c &d \end{psmallmatrix}$ and let $\tau \in \Omega$. We have $|c|  \geq 1$ and by Lemma \ref{lem:a-less-than-c} and \eqref{eq:surprisingly-useful-identity}, 
\[
|M\tau|= \left| \frac{a}{c}- \frac{1}{c(c\tau + d)} \right| \leq 1 + \frac{1}{|c\tau +d|} \leq 1 +  \frac{1}{\inf_{z \in \Omega}{\mathfrak{s}(\Lambda_{z})}},
\]
which is finite and depends only on $\Omega$. 
\end{proof}
\begin{remark-non}
In the above proof, instead of using continuity of the shortest vector function, one can simply use that $|c\tau + d| \geq \imag(\tau)$, for $c \neq 0$ and thus generalize the corollary to subsets $\Omega$ satisfying $\inf_{\tau \in \Omega}{(\imag(\tau))} >0$.
\end{remark-non}
\subsection{Definition of the generating series and the  basis functions}\label{sec:definition-of-generating-series}
With the preparations from the previous subsections, we are now ready to give solutions to the functional equations in \S \ref{sec:generating-series-and-functional-equations} and give the definition of the functions $b_{p,n}, \tilde{b}_{p,n}$ entering Theorem \ref{thm:main-thm-radial-poincare}.  Let $p \geq 5$ be an integer. For $\tau \in \H$ and $r \in \C$ define the series
\begin{align}
F_p(\tau,r) &= -\sum_{M \in \mathcal{B}}{e^{\pi i \tau r^2}|_{p/2}M} =-\sum_{M \in \mathcal{B}}{(\Theta_3(M\tau)/\Theta_3(\tau))^{-p} e^{\pi i r^2 M\tau }}, \label{eq:def:Fp} \\
\tilde{F}_p(\tau,r) &= \sum_{M \in \tilde{\mathcal{B}}}{e^{\pi i \tau r^2}|_{p/2}M} =\sum_{M \in \tilde{\mathcal{B}}}{(\Theta_3(M\tau)/\Theta_3(\tau))^{-p} e^{\pi i r^2 M\tau }}. \label{eq:def:tildeFp}
\end{align}
We now show they converge absolutely and uniformly on compact sets. So let $\Omega_1 \subset \H$ and $\Omega_2 \subset \C$ be compact subsets. Then by \eqref{eq:automorphic-factors-the-same} and by Corollary \ref{cor:a-less-than-c}, we have, for all $M \in \mathcal{B} \cup \tilde{\mathcal{B}}$, $(\tau,r) \in \Omega_1 \times \Omega_2$,
\[
\left| (\Theta_3(M\tau)/\Theta_3(\tau))^{-p} e^{\pi i r^2 M\tau } \right| \leq  \frac{\exp{(\pi |r|^2 |M\tau |)}}{|c_M \tau + d_M|^{p/2}} \ll_{\Omega_1, \Omega_2} \frac{1}{|c_M \tau + d_M|^{p/2}},
\]
By compactness, there exists $C =C_{\Omega_1} > 0$ with the property that  $|c_M i + d_M| \leq C|c_M \tau + d_M|$ for all $M \in \mathcal{B} \cup \tilde{\mathcal{B}}$ and all $\tau \in \Omega_1$. We deduce
\begin{equation}\label{eq:uniform-summand-bound-for-F-p}
\sup_{(\tau,r) \in \Omega_1 \times \Omega_2}{\left|(\Theta_3(M\tau)/\Theta_3(\tau))^{-p} e^{\pi i r^2 M\tau  }\right|} \ll_{\Omega_1, \Omega_2} \frac{1}{|c_M i + d_M|^{p/2}},
\end{equation}
where the implied constant does not depend upon $M \in \mathcal{B} \cup \tilde{\mathcal{B}}$. Since $ p \geq 5$, the sequence  $( \sum_{0 < c^2 + d^2 \leq N}{|ci+d|^{-p/2}})_{N \in \N}$ is bounded and increasing in $[0, \infty)$, which, combined with \eqref{eq:uniform-summand-bound-for-F-p} and the injectivity of the mappings in Lemma \ref{lem:identifying-sets-B-tilde-B}, implies that the series defining  $F_p, \tilde{F}_p$ converge pointwise absolutely and uniformly on $\Omega_1 \times \Omega_2$ and thus define continuous functions on $\H \times \C$ that are holomorphic in each variable separately.
  
Part (ii) of Lemma \ref{lem:identifying-sets-B-tilde-B} asserted that $\mathcal{B}$ and $\tilde{\mathcal{B}}$ are stable under right multiplication by powers of $A$. By absolute convergence, we deduce that the functions $F_p$, $\tilde{F}_p$ are  both $2$-periodic in the first argument. By definition of the set  $\tilde{\mathcal{B}}$ and because of the minus sign in the definition of $F_p(\tau,r)$, they are moreover related by the functional equation 
\begin{equation}
F_p(\tau,r) + (-i \tau)^{-p/2}\tilde{F}_p(-1/\tau) = e^{\pi i r^2 \tau}.
\end{equation}
Replacing $r$ by the Euclidean norm of $x \in \R^p$, gives the desired solutions to the system of functional equations in \S \ref{sec:generating-series-and-functional-equations}. For $n \in \Z$, we define
\begin{align}
b_{p,n}(r) &= \frac{1}{2} \int_{iy_0 + [-1,1]}{F_p(\tau, r) e^{- \pi i n \tau} d \tau}, \label{eq:def-b-p-n}\\
 \tilde{b}_{p,n}(r) &= \frac{1}{2} \int_{iy_0 + [-1,1]}{ \tilde{F}_p(\tau, r) e^{- \pi i n \tau} d \tau}, \label{eq:def-tilde-b-p-n}
\end{align}
for any $y_0 > 0$, as the integrals are independent of $y_0$ (\S \ref{sec:two-periodic-hol-functions}). By continuity of $F_p$ and $\tilde{F}_p$ and holomorphy in the second argument, the functions $r \mapsto b_{p,n}(r)$ and $r \mapsto \tilde{b}_{p,n}(r)$ are entire and they are clearly even. By the general remarks of \S \ref{sec:notation-radial-near-origin}, the functions $x \mapsto b_{p,n}(|x|)$, $x \mapsto  \tilde{b}_{p,n}(|x|)$ are smooth on $\R^d$, but we will also prove this directly in the next section. 
\subsection{Upper bounds for Fourier coefficients}\label{sec:growth-estimates-poincare}
To complete our implementation of the general strategy explained in \S \ref{sec:generating-series-and-functional-equations} and thus prove Theorem 2, we must give upper bounds for the Fourier coefficients $b_{p,n}(r)$ and $\tilde{b}_{p,n}(r)$ defined in \eqref{eq:def-b-p-n}, \eqref{eq:def-tilde-b-p-n} in terms of $n$ and $r$. We will do so by first bounding the generating functions $F_p(\tau,r)$, $\tilde{F}_{p}(\tau,r)$ themselves and then applying the triangle inequality to the integrals for a suitable height $y_0 > 0$. In the end, we will take $y_0 \asymp p/n$, but also want the upper bound to hold for \emph{all}  pairs $(n,p)  \in\N \times \Z_{\geq 5}$, since we implicitly sum over them in our main interpolation formula. We therefore seek bounds for $F_p(\tau,r)$ and $\tilde{F}_p(\tau,r)$ that are equally uniform in $y_0 = \imag(\tau)$. To this end, we define, for any real $k > 2$, the auxiliary functions $U_k, \tilde{U}_k : \H \rightarrow (0,+\infty)$ by
\begin{equation}\label{eq:definition-U-tilde-U}
U_k(\tau) = \sum_{M \in \mathcal{B}}{|c_M \tau + d_M|^{-k}}, \quad \tilde{U}_k(\tau) = \sum_{M \in \tilde{\mathcal{B}}}{|c_M \tau + d_M|^{-k}}.
\end{equation}
Note that $|F_p(\tau,r)| \leq U_{p/2}(\tau)$ and $|\tilde{F}_p(\tau,r)| \leq \tilde{U}_{p/2}(\tau)$ for all $(\tau,r) \in \H \times \R$.
\begin{lemma}\label{lem:upper-bounds-for-U}
There exists a constant $C_0 > 0$ with the following property. For all $\varepsilon  \in (0,1/8]$, all $k \geq 2+2\varepsilon$, all $x \in [-1,1]$ and all $y_0 >0$, we have
\[
\max{(U_{k}(x+iy_0), \tilde{U}_{k}(x+iy_0))   } \leq C_0 \varepsilon^{-2} (y_0^{-k} + y_0^{-k/2}).
\]
\end{lemma}
\begin{proof}
By absolute convergence and the injectivity assertions from Lemma \ref{lem:identifying-sets-B-tilde-B} and by simply enlarging the sets $\mathcal{P}$, $\tilde{\mathcal{P}}$,  we have
\[
\max{(U_{k}(x+iy_0), \tilde{U}_{k}(x+iy_0))   } \leq \sum_{c=1}^{\infty}{ \sum_{d=1}^{c}{  \sum_{\ell \in \Z}{  \frac{1}{ \big((cx+d+\ell c)^2 + (cy_0)^2 \big)^{k/2}}   }  }   }.
\]
To bound the denominators from below, we first write 
\[
(cx+d+\ell c)^2 + (cy_0)^2 = c^2 \big(( x+d/c + \ell)^2+ y_0^2\big)
\]
and then use, in the range $|\ell| \leq 2$, the trivial estimate 
\[
( x+d/c + \ell)^2+ y_0^2 \geq y_0^2,
\]
while in the range $|\ell| \geq 3$, we use 
\[
( x+d/c + \ell)^2+ y_0^2 \geq 2|x+d/c + \ell| y_0 \geq 2(|\ell| -2)y_0,
\]
which holds since $|x| \leq 1$ and $|d/c| \leq 1$ for all terms in the series. We deduce that
\[
\max{(U_{k}(x+iy_0), \tilde{U}_{k}(x+iy_0))   } \leq \sum_{c=1}^{\infty}{ c^{1-k} \Big(5 y_0^{-k} + (2y_0)^{-k/2} \sum_{|\ell| \geq 3}{(|\ell|-2)^{-k/2}} \Big)   },
\]
which is now a product. For $s>1$, let $\zeta(s) = \sum_{n=1}^{\infty}{n^{-s}}$. The sum over $|\ell| \geq 3$ is at most $2\zeta(1+ \varepsilon)$, while the sum over $c$ is at most $\zeta(1+ 2\varepsilon)$. We conclude the analysis by recalling that $\lim_{s \rightarrow 1}{(s-1) \zeta(s)} = 1$. 
\end{proof}
\begin{corollary}\label{cor:b_p-zero-for-negative-n}
If $n \leq 0$ then $b_{p,n} = 0= \tilde{b}_{p,n}$.
\end{corollary}
\begin{proof}
By analyticity, it suffices to show that  $b_{p,n}(r) = 0= \tilde{b}_{p,n}(r)$ for all $r \in \R$. By Lemma \ref{lem:upper-bounds-for-U} and \eqref{eq:def-b-p-n} we have, for all $y_0 > 0$ and $r \in \R$,
\[
|b_{p,n}(r)| \leq e^{\pi n y_0} \sup_{\tau \in iy_0 +[-1,1]}{U_{p/2}(\tau)}  \ll e^{\pi n y_0} \big( y_0^{-p/2} + y_0^{-p/4} \big),
\]
where the implied constant is independent of $p, n,r$ and $y_0$.
Since $e^{\pi n y_0} \leq 1$ we can let $y_0 \rightarrow \infty$ to deduce $b_{p,n}(r) = 0$. The argument for $\tilde{b}_{p,n}$ is very similar.
\end{proof}
In the remainder of \S \ref{sec:growth-estimates-poincare}, we prove assertions (i) and (ii) of Theorem \ref{thm:main-thm-radial-poincare}, that is, we prove the claimed upper bounds for $\partial^{\alpha}b_{p,n}(|x|)$, $\partial^{\alpha}\tilde{b}_{p,n}(|x|)$ for $x \in \R^d$ and $\alpha \in \N_0^d$. In view of Corollary \ref{cor:b_p-zero-for-negative-n} and the general remarks of \S \ref{sec:generating-series-and-functional-equations}, this will then also prove the radial interpolation formula \eqref{eq:interpolation-formula-radial-in-main-thm-poincare} and complete the proof of Theorem \ref{thm:main-thm-radial-poincare}.  We focus on the analysis of $F_p$ and $b_{p,n}$; the one for $\tilde{F}_p$ and $\tilde{b}_{p,n}$ is the same, because of the maximum in Lemma \ref{lem:upper-bounds-for-U}.  We work with the following parameters and notations.
\begin{itemize}
\item A real number $\varepsilon \in (0, 1/8]$.
\item A constant $C_0>0$ having the property stated in Lemma \ref{lem:upper-bounds-for-U}. Until the end of \S \ref{sec:growth-estimates-poincare}, a constant will be called \emph{absolute}, if it depends at most $C_0$.
\item For each $0 \leq j \leq |\alpha|$, the polynomial $P_j = P_{\alpha, d, j} \in \Z[ 2 \pi i][x_1, \dots, x_d]$ of degree at most $|\alpha|$ with the property that for all $z \in \C$ and $x \in \R^d$,
\begin{equation}\label{eq:partial-derivative-gaussian}
\partial_x^{\alpha}{e^{\pi i z|x|^2}} =  e^{\pi i z |x|^2} \sum_{j=0}^{|\alpha|}{P_j(x)z^j}.
\end{equation}
These will play no role if $\alpha = 0$, a case worth focusing on in  a first reading.
\item The parameter $\sigma = \sigma_{p, \varepsilon} = p/4-(1+ \varepsilon) \geq 1/8$.
\item For $|x|>0$, the shorthand $B_{\sigma}(|x|) = \left(\frac{\sigma}{\pi e |x|^2} \right)^{\sigma} = \sup_{y \in (0, + \infty)}{y^{\sigma}e^{-\pi y|x|^2}}$.
\end{itemize}
To start, we differentiate \eqref{eq:def-b-p-n}, giving 
\begin{equation}\label{eq:differentiating-bpn-under-integral}
\partial_x^{\alpha}b_{p,n}(|x|) =  \frac{1}{2}\int_{iy_0 + [-1,1]}{\partial_x^{\alpha}F_p(\tau,|x|) e^{- \pi  i n \tau }d \tau}.
\end{equation}
To bound $\partial_x^{\alpha}F_p(\tau,|x|)$, we apply \eqref{eq:partial-derivative-gaussian} with $z = M\tau = \frac{a_M}{c_M}- \frac{1}{c_M(c_M\tau + d_M)}$ and obtain
\[
\partial_x^{\alpha}{e^{\pi i (M\tau)|x|^2}} =  e^{\pi i (M\tau) |x|^2} \sum_{j=0}^{|\alpha|}{P_j(x) \sum_{t=0}^{j}{\binom{j}{t} (a_M/c_M)^{j-t} (-c_M(c_M\tau + d_M))^{-t} }}.
\]
We have $|c_M| \geq 1$ by Lemma \ref{lem:identifying-sets-B-tilde-B} and  $|a_M/c_M| \leq 1$ by Lemma \ref{lem:a-less-than-c}, hence
\[
|\partial_x^{\alpha}{e^{\pi i (M \tau) |x|^2}}| \leq  e^{-\pi \imag( M\tau) |x|^2} \sum_{j=0}^{|\alpha|}{|P_j(x)| \sum_{t=0}^{j}{\binom{j}{t}}{ |c_M \tau +d_M|^{-t}   }}.
\]
We may now either use the trivial bound $e^{-\pi \imag( M\tau) |x|^2} \leq 1$, or, if $|x|>0$, 
\begin{align*}
e^{-\pi \imag( M\tau) |x|^2} = \imag(M\tau)^{\sigma}e^{-\pi \imag( M\tau) |x|^2} \imag(M \tau)^{- \sigma}  \leq B_{\sigma}(|x|) |c_M\tau +d_M|^{2 \sigma} \imag(\tau)^{-\sigma}.
\end{align*}
Using the auxiliary function $U_k$, defined in \eqref{eq:definition-U-tilde-U}, we deduce
\begin{align}
|\partial_x^{\alpha}F_p(\tau,|x|)| &\leq  \sum_{j=0}^{|\alpha|}{|P_j(x)| \sum_{t=0}^{j}{\binom{j}{t}}{ U_{p/2+t}(\tau)   }},  \quad \text{(from the trival bound)}\label{eq:bound-partial-F_p-uniform-detailled}  \\
 |\partial_x^{\alpha}F_p(\tau,|x|)| &\leq  B_{\sigma}(|x|)\imag(\tau)^{-\sigma}\sum_{j=0}^{|\alpha|}{|P_j(x)| \sum_{t=0}^{j}{\binom{j}{t}}{ U_{p/2-2\sigma+t}(\tau)   }}, \quad  \text{if }  |x|>0. \label{eq:bound-partial-F_p-away-form-zero} 
\end{align}
We now apply the triangle inequality to \eqref{eq:differentiating-bpn-under-integral} and use Lemma \ref{lem:upper-bounds-for-U}, applied with $k = p/2+t$ and the binomial theorem (read ``backwards"), to deduce from  \eqref{eq:bound-partial-F_p-uniform-detailled} that
\begin{equation}
|\partial_x^{\alpha}b_{p,n}(|x|)| \leq  64 C_0 e^{\pi n y_0} \sum_{j=0}^{|\alpha|}{|P_j(x)| \left( y_0^{-p/2}(1+y_0^{-1})^j + y_0^{-p/4}(1+y_0^{-1/2})^{j} \right) }.\label{eq:bound-partial-b_p-uniform-detailled}
\end{equation}
If $y_0 = \frac{p}{2\pi n}$ (so that $1/y_0 \leq 2n$), then  \eqref{eq:bound-partial-b_p-uniform-detailled} implies (after some calculations) 
\begin{equation}
|\partial_x^{\alpha}b_{p,n}(|x|)| \leq H_1 n^{p/2} (2\pi e^2/p)^{p/4}\sum_{j=0}^{|\alpha|}{|P_j(x)| (1+2n)^j },\label{eq:bound-partial-b_p-uniform-detailled-final}
\end{equation}
for some absolute constant $H_1 >0$. We deduce similarly from \eqref{eq:bound-partial-F_p-away-form-zero} and Lemma \ref{lem:upper-bounds-for-U}, applied with $k= p/2-2 \sigma + t =2+2 \varepsilon +t \geq 2+2 \varepsilon$, that $|\partial_x^{\alpha}b_{p,n}(|x|)|$ is less than or equal to
\begin{equation}
\varepsilon^{-2} C_0 B_{\sigma}(|x|) y_0^{-\sigma}  e^{\pi n y_0} \sum_{j=0}^{|\alpha|}{|P_j(x)| \left( y_0^{-2(1+\varepsilon)}(1+y_0^{-1})^j + y_0^{-(1+ \varepsilon)}(1+y_0^{-1/2})^{j} \right)    },\label{eq:bound-partial-b_p-away-from-zero-detailled}
\end{equation}
if $|x|>0$. If $y_0 = \frac{\sigma}{\pi n}$ (so that $1/y_0 \leq 30n$) then \eqref{eq:bound-partial-b_p-away-from-zero-detailled} implies (after some calculations)
\begin{equation}
|\partial_x^{\alpha}b_{p,n}(|x|)| \leq H_2 \varepsilon^{-2} n^{p/4+1+ \varepsilon}|x|^{-p/2+2(1+\varepsilon)} \sum_{j=0}^{|\alpha|}{|P_j(x)|(1+30 n)^j}, \label{eq:bound-partial-b_p-away-from-zero-detailled-final}
\end{equation}
for some absolute constant $H_2$. Here, the choice of $y_0$ also ensured that the term $(\sigma/(\pi e))^{\sigma}$ coming from $B_{\sigma}(|x|)$ disappeared. To obtain the final bounds in Theorem \ref{thm:main-thm-radial-poincare}, it only remains to bound the polynomials $|P_j(x)|$ for $|x| \leq R$ by compactness and continuity (which we do only when $\alpha \neq 0$) and to use $\sum_{j=0}^{|\alpha|}{(1+\kappa n)^j}  \ll_{\kappa,|\alpha|} n^{|\alpha|}$, for $\kappa \in \{2, 30\}$.

\section{Other function spaces}\label{sec:extension-to-non-schwartz}
Here we extend Theorem \ref{thm:main-thm-non-radial} from $\mathcal{S}(\R^d)$ to a larger function space. We closely follow the approach of \cite[Prop. 4]{R-V}, which generalizes to higher dimensions without much difficulty.

\subsection{Preliminaries}\label{sec:notation-function-spaces}
For any $k \in \N_0$, we denote by $C^{k}(\R^d)$ the space of $k$-times continuously differentiable functions $f: \R^d \rightarrow \C$ whose partial derivatives are all bounded on $\R^d$. For $f \in C^{k}(\R^d)$ we denote its $C^k$-norm by $\|f\|_{C^k(\R^d)} = \sum_{|\alpha| \leq k}{\sup_{x \in \R^d}{|\partial^{\alpha}f(x)|}}$. For \emph{every} function $f : \R^d \rightarrow \C$ and \emph{every} $B > 0$, we define the extended real number \[
Q_B(f) = \sup_{x \in \R^d}{\left( (1+|x|^{B})|f(x)| \right)} \in [0,+ \infty]
\] 
and then, for every $B > d$,  the space
\begin{equation}\label{eq:def-WL}
\mathcal{W}_B(\R^d) = \{ f \in C^0(\R^d) \,:\, Q_B(f) < \infty,\, Q_B(\hat{f}) < \infty \}.
\end{equation}
Note that if $B > d+2$ and $f \in \mathcal{W}_B(\R^d)$, then $f \in C^2(\R^d)$. The next Lemma shows that we can then also control the decay of the first-order partial derivatives of $f$.
\begin{lemma}\label{lem:decay-gradient}
Let $B > 0$ and $f \in C^2(\R^d)$.  Then $Q_B(f) < \infty$ implies $Q_{B/2}(|\nabla f|)< \infty$.
\end{lemma}
\begin{proof}[Proof sketch]
Suppose that $Q_B(f) < \infty$. For $y \in \R^d$, denote by $H_f(y)$ the Hessian of $f$ at $y$. Then by Taylor's theorem we have, for any $x, \xi \in \R^d$, 
\[
f(x + \xi) = f(x) + \xi \cdot \nabla f(x) + \int_{0}^1{(1-t) \left(\xi  \cdot H_f(x + t\xi) \xi \right)dt}.
\] 
By assumption, $y \mapsto H_f(y)$ is a continuous bounded function on $\R^d$. Hence from the above,
\begin{equation}\label{eq:taylor-expansion-order-2-qualitative}
\xi \cdot \nabla f(x) = f(x)-f(x + \xi) + O(|\xi|^2).
\end{equation}
Fixing $x \in \R^d$ with $|x| \geq 1$ and taking $\xi =  \varepsilon \nabla f(x)$ with $\varepsilon > 0$ chosen small enough in terms of the implied constant in \eqref{eq:taylor-expansion-order-2-qualitative} and $\sup_{\R^d}{|\nabla f|}$ , we conclude.
\end{proof}
\subsection{Convolutions }
We fix a dimension $d \geq 1$ and write
\[
\phi(x) = e^{- \pi |x|^2}, \qquad \phi_{\varepsilon}(x) = \phi(x/ \varepsilon) \varepsilon^{-d}, \qquad 
\psi_{\varepsilon}(x) = \phi( \varepsilon x)
\]
for the Gaussian, the Gaussian approximate identity and the ``flat" Gaussian respectively, where $\varepsilon > 0$ and $x \in \R^d$. We have $\widehat{\phi_{\varepsilon}} =\psi_{\varepsilon}$ and $ \widehat{ \psi_{\varepsilon}} = \phi_{\varepsilon}$. For any $f,g \in C^0(\R^d)$, we define 
\begin{equation}\label{def:convolution-operators}
J_{\varepsilon}f = \psi_{\varepsilon} \cdot ( f \ast \phi_{\varepsilon}), \qquad \tilde{J}_{\varepsilon}g = \phi_{\varepsilon}  \ast ( g  \cdot \psi_{\varepsilon}).
\end{equation}  
For every subset $\Omega \subset \R^d$ and every $r \geq 0$, we write $B_{r}(\Omega)$ to denote the set of all $x \in \R^d$, for which there exists $\omega \in \Omega$ such that $|x- \omega| \leq r$. We write $B_r(x) = B_r(\{x\})$ for $x \in \R^d$.
\begin{lemma}\label{lem:convolution-ops}
The operators $J_{\varepsilon}, \tilde{J}_{\varepsilon}$ have the following properties.
\begin{enumerate}[(i)]
\item For every $f \in C^0(\R^d)$ and all $\varepsilon >0$, we have $J_{\varepsilon}f \in \mathcal{S}(\R^d)$
\item For all $B > d$, all $f \in \mathcal{W}_B(\R^d)$ and all $\varepsilon >0$, we have $\widehat{J_{\varepsilon}f} = \tilde{J}_{\varepsilon}\hat{f}$.
\item There exists a constant $C_1>0$, depending only on $d$, such that for all $f \in C^1(\R^d)$, all $x \in \R^d$, and all $\varepsilon > 0$, we have
\[
|J_{\varepsilon}f(x) -f(x)| \leq C_1  e^{- \pi |\varepsilon x|^2} \left( \varepsilon \sup_{B_1(x)}{|\nabla f|} +  e^{-\frac{\pi}{2 \varepsilon^2}}\|f\|_{C^0(\R^d)}  \right) +C_1  \varepsilon^2  |x|^2 |f(x)| .
\]
\item There exists a constant $C_2 >0$, depending only on $d$, such that for all $g \in C^1(\R^d)$, all $\xi \in \R^d$ satisfying $|\xi| \geq 1$ and all $\varepsilon  \in (0,1]$, we have
\[
|{\tilde{J}}_{\varepsilon}g(\xi) -g(\xi)| \leq C_2\left( \varepsilon \sup_{B_{|\xi|/2}(\xi)}{|\nabla g|} + e^{- (\pi/8) |\xi/ \varepsilon|^2} \|g\|_{C^0(\R^d)} + \varepsilon^2 |\xi|^2 |g(\xi)|  \right).
\]
\end{enumerate}
\end{lemma}
\begin{proof}
We believe this to be standard, but we sketch the proof for completeness. For (i), we readily check that $\phi_{\varepsilon} \ast f$ is smooth with bounded derivatives. For (ii), we recall that for $B > d$ we have $\mathcal{W}_B(\R^d) \hookrightarrow L^1(\R^d)$, so that the claim follows from $\widehat{\phi_{\varepsilon}} = \psi_{\varepsilon}$ and the convolution theorem. To prove (iii), we write $J_{\varepsilon}f(x)-f(x) = X + Y + Z$, where:
\begin{align*}
X &= \psi_{\varepsilon}(x) \int_{|y| \leq 1}{\phi_{\varepsilon}(y)\int_{0}^{1}{ (\nabla f (x+t y) \cdot y)   dt}\,dy},\\
Y &= \psi_{\varepsilon}(x) \int_{|y| \geq 1}{\phi_{\varepsilon}(y)  \left( f(x+y)-f(x) \right) dy},\\
 Z &= (\psi_{\varepsilon}(x)-1)f(x).
\end{align*}
The integral $X$ gives the first term in the inequality claimed in (iii), where the factor $\varepsilon$ comes from a change of variables $y \leftrightarrow y/ \varepsilon$. The integral $Y$ gives the second,  using $\int_{|y| \geq 1}{\phi_{\varepsilon}(y)} \ll_{d} e^{- \tfrac{\pi}{2 \varepsilon^2}}$. The integral $Z$ gives the third, using $|\psi_{\varepsilon}(x)-1| \leq \pi \varepsilon^2 |x|^2$. To prove (iv), suppose that $|\xi| \geq 1$ and write $\tilde{J}_{\varepsilon}g(\xi)- g(\xi) = U+V+W$, where
\begin{align*}
U &= \int_{|y| \leq |\xi|/2}{\phi_{\varepsilon}(y) \psi_{\varepsilon}(\xi + y)   \int_{0}^{1}{( \nabla g(\xi + t y) \cdot y )dt} \, dy},\\
V &=  \int_{|y| \geq |\xi|/2}{\phi_{\varepsilon}(y) \psi_{\varepsilon}(\xi + y) \left( g(y+ \xi) - g(\xi) \right) dy}, \\
W &= g(\xi) \int_{\R^d}{\phi_{\varepsilon}(y)  \left( \psi_{\varepsilon}(y+\xi)-1 \right) dy}.
\end{align*}
To bound $U$, we use the gradient bound as for $X$. For $V$, we first apply the triangle inequality and then change to the variable $u = y/ \varepsilon$, to obtain
\[
|V| \leq 2\|g\|_{C^{0}(\R^d)} \int_{|y| \geq \tfrac{|\xi|}{2}}{\phi(y/\varepsilon) \psi_{\varepsilon}(y + \xi) \varepsilon^{-d}dy} = 2\|g\|_{C^{0}(\R^d)}\int_{|u| \geq \tfrac{|\xi|}{2 \varepsilon}}{\phi(u) \psi_{\varepsilon}(\varepsilon u + \xi) du}.
\]
Writing $\phi(u) = e^{- \pi |u|^2/2} e^{- \pi |u|^2/2}$ and bounding $\psi_{\varepsilon}(\varepsilon u + \xi) \leq 1$ here, we get the second term claimed in (iv).  For $W$, we apply the triangle inequality and use the estimate
\[
|\psi_{\varepsilon}(y + \xi) - 1| \leq \pi \varepsilon^2 |y+ \xi|^2 \leq \pi \varepsilon^2 |\xi|^2 (|y|+1)^2,
\]
where the last inequality uses the assumption $|\xi| \geq 1$. We bound bound the remaining integral independently of $\varepsilon$,  by changing to the variable $u = y/ \varepsilon$, noting that $(|\varepsilon u|+1)^2 \leq (|u| + 1)^2$, since $\varepsilon \leq 1$.
\end{proof}
\subsection{Limiting argument}\label{sec:limit-argument}
Suppose that  $d \geq 5$ and $A_n, \tilde{A}_n  \in C^{\infty}(\R^d \times S)$ be such that they satisfy the conclusion of Theorem \ref{thm:main-thm-non-radial}. In principle, a similar discussion applies to lower dimensions, using Theorem \ref{thm:non-radial-small-dimensions}, but we stick to $d\geq 5$ for simplicity. 

We consider henceforth a fixed compact  subset $\Omega \subset \R^d$ and we suppose given constants $K, a, c > 0$ so that for all $n \in \N$,
\begin{equation}\label{eq:bound-An-annulus}
\sup_{(x, \zeta) \in \Omega \times S}{\left(|A_n(x, \zeta)|+|\tilde{A}_n(x, \zeta)|\right)} \leq K n^{ad +c}.
\end{equation}
If $A_n$, $\tilde{A}_n$ are as defined in \S \ref{sec:proof-main-theorem}, then Theorem \ref{thm:main-thm-radial-poincare} and Lemma \ref{lem:A-inf-summands-bounds} provide admissible values of $a,c$. Namely, one can take $(a,c) = (1/2,0)$ if $\Omega =  \{0 \}$, or $(a,c) = (5/4,1/8)$ if $0 \notin \Omega$. We proceed generally and specialize to these  values later. Consider a decay rate $B$ satisfying 
\begin{equation}\label{eq:L-constraint}
B > \max{(d+2, 4(1+ad+c))}.
\end{equation}
For all $f, g \in C^0(\R^d)$, satisfying $Q_B(f) < \infty $ and $Q_B(g) < \infty$ and all $x \in \Omega$, we may define
\begin{align*}
Rf(x) = \sum_{n=1}^{\infty}{\int_S{A_n(x, \zeta) f(\sqrt{n}\zeta) d\zeta}}, \qquad 
\tilde{R}g(x) = \sum_{n=1}^{\infty}{\int_S{{\tilde{A}}_n(x, \zeta) g(\sqrt{n}\zeta) d\zeta}},
\end{align*}
which converge absolutely and vary continuously with $x \in \Omega$, since $B > 2(1+ad+c)$. Let $f \in \mathcal{W}_B(\R^d)$. It follows from parts (i) and (ii) of Lemma \ref{lem:convolution-ops}  and from Theorem \ref{thm:main-thm-non-radial} that for all $\varepsilon > 0$,
\begin{align*}
f &= (f-J_{\varepsilon}f) + J_{\varepsilon}f 
= (f-J_{\varepsilon}f) + R(J_{\varepsilon}f) + \tilde{R}(\widehat{J_{\varepsilon}f})\\
&= (f-J_{\varepsilon}f) + (Rf + \tilde{R} \hat{f}) +  R(J_{\varepsilon}f-f) + \tilde{R}(\tilde{J}_{\varepsilon} \hat{f}- \hat{f}),
\end{align*}
as functions on $\Omega$. We want to show that $f  = Rf + \tilde{R} \hat{f}$, so it suffices to show that the terms depending upon $\varepsilon$ tend to zero as $\varepsilon$ tends to zero. By part (iii) of Lemma \ref{lem:convolution-ops}, we have $ \sup_{\Omega}{|f- J_{\varepsilon}f|} \rightarrow 0$, as $\varepsilon \rightarrow 0$ and our assumption \eqref{eq:bound-An-annulus} implies
\begin{align}
\sup_{\Omega}{|R(J_{\varepsilon}f-f)|} &\leq K \sum_{n=1}^{\infty}{ n^{ad+c} \sup_{\sqrt{n}S}{|J_{\varepsilon}f-f|}  }, \label{eq:sup-1}\\
\sup_{\Omega}{|\tilde{R}(\tilde{J}_{\varepsilon}\hat{f}-\hat{f})|} &\leq K \sum_{n=1}^{\infty}{ n^{ad+c} \sup_{\sqrt{n}S}{|\tilde{J}_{\varepsilon} \hat{f}-\hat{f}|}  }.\label{eq:sup-2} 
\end{align}
It follows from part (iii) of Lemma \ref{lem:convolution-ops}, applied with $x = \zeta \sqrt{n}$ and part (iv) with $\xi = \zeta \sqrt{n}$, for $(\zeta, n) \in S \times \N$ and the assumption on the decay rate $B$, that  \eqref{eq:sup-1}, \eqref{eq:sup-2} are both $O(\varepsilon)$. Here, the more subtle terms come from the gradients of $f$ and $\hat{f}$, which may be controlled by Lemma \ref{lem:decay-gradient}, implying the estimates
\[
\sup_{B_1(\sqrt{n}S)}{|\nabla f|} \ll n^{-B/4}, \qquad \sup_{B_{\sqrt{n}/2}(\sqrt{n}S)}{|\nabla \hat{f}|} \ll n^{-B/4}.
\]
To summarize, assuming the bound \eqref{eq:bound-An-annulus} on $A_n$, $\tilde{A}_n$ and assuming $B$ satisfies \eqref{eq:L-constraint}, the interpolation formula \eqref{eq:ultimate-interpolation-formula-in-thm} holds for all $f \in \mathcal{W}_B(\R^d)$ and all $x \in \Omega$ with uniform convergence. Specializing the discussion to the concrete values $(a,c) = (5/4, 1/8)$ and noting that $5/4 > 1/2$ and $4(1+5d/4 + 1/8) = 5d+9/2$, we obtain the following corollary. 
\begin{corollary}\label{cor:extension}
Suppose that $B > 5d+9/2$. Then the interpolation formula \eqref{eq:ultimate-interpolation-formula-in-thm} in Theorem \ref{thm:main-thm-non-radial} holds for all $f \in \mathcal{W}_B(\R^d)$ with absolute convergence at every point and uniform convergence on compact subsets avoiding the origin. 
\end{corollary}

\section{Relations between restrictions of Schwartz functions to spheres}\label{sec:space-of-relations-infinite-dimensional}
Here we elaborate on the remarks on free interpolation made in \S \ref{sec:intro-context}. The main result of this section, Proposition \ref{prop:relations-between-restrictions} below, won't be used elsewhere in the paper, but may give an interesting comparison to other work. We again restrict to dimensions $d \geq 5$ for simplicity. 

Recall that Radchenko and Viazovska prove in \cite[Thm2]{R-V}, that the linear map sending $f \in \mathcal{S}_{\text{rad}}(\R^1)$ to the pair of sequences $(f(\sqrt{n}))_{n \in \N_0}$, $(\hat{f}(\sqrt{n}))_{n \in \N_0}$ defines an isomorphism of Fr{\'e}chet spaces with a subspace of co-dimension one, in the space of all pairs of rapidly decreasing sequences of complex numbers. This subspace is cut out by a single linear functional coming from Poisson summation. 

In our setting of not necessarily radial functions, we consider the linear map
\begin{equation}\label{eq:restriction-to-spheres-sequence-mapping}
\Phi_d : \mathcal{S}(\R^d) \longrightarrow \mathcal{V}_d, \qquad f \mapsto \left( ( f( \sqrt{n}\,  \cdot ))_{n \in \N}, ( \hat{f}(\sqrt{n}\, \cdot ))_{n \in \N} \right),
\end{equation}
where $\mathcal{V}_d$ denotes the  the space all pairs of sequences of functions $f_n, g_n \in C^{\infty}(S^{d-1})$, whose sup-norms decay rapidly with $n$.  
\begin{proposition}\label{prop:relations-between-restrictions}
For $d \geq 5$, the map $\Phi_d$ has infinite dimensional cokernel. In fact, the annihilator of the image of $\Phi_d$ is an infinite dimensional subspace of the dual space $\mathcal{V}_d^{\ast}$.
\end{proposition}
To prepare the proof of this proposition, let us introduce the theta functions
\[
\Theta_2(\tau) = \theta_{10}(\tau) =   \sum_{n \in \Z}{e^{\pi i (n+1/2)^2 \tau}}, \qquad 
 \Theta_4(\tau) = \theta_{01}(\tau) = \sum_{n \in \Z}{ (-1)^n e^{\pi i n^2 \tau}}.
\]
For any half-integer $k \geq 0$, let $M_k(\Gamma(2))$  denote the space of modular forms of weight $k$ for $\Gamma(2)$, where modularity refers to the slash action introduced in \S \ref{sec:slash-action}. By \cite[Thm 7.1.7]{rankin-book}, this space has dimension $1+\lfloor k/2 \rfloor$ and $\{ \Theta_2^{4j}\Theta_3^{2k-4j}\}_{0 \leq j \leq \lfloor k/2 \rfloor }$ is a basis. For $\varphi \in M_k(\Gamma(2))$ we define $\varphi_0 \in M_k(\Gamma(2))$ by  $\varphi_0(\tau)= (-i \tau)^{-k}\varphi(-1/\tau)$.

We moreover fix, for each $m \geq 0$, an orthonormal basis $\mathcal{B}_m \subset \mathcal{H}_m(S^{d-1})$ and define the auxiliary function $P_m(\zeta) = \sum_{u \in \mathcal{B}_m}{ \overline{u(\zeta)}}$, Note that $\langle P_m ,P_{\mu} \rangle_{L^2(S)} = \delta_{m, \mu} |\mathcal{B}_m|$.
\begin{proof}[Proof of Proposition \ref{prop:relations-between-restrictions}]
Recall first that the spaces $M_{k}(\Gamma(2))$ are linearly independent as $k$ varies. We will define a linear map $\varphi \mapsto \varphi^{\ast}$ from the space $\mathcal{M}_d = \oplus_{m \geq 0}{M_{d+m/2}(\Gamma(2))}$ to the annihilator of the image of $\Phi_d$ and show that this map restricts to an injection on the infinite dimensional subspace $\mathcal{J}_d$, consisting of all finite sums of forms $\varphi \in M_{d/2+m}(\Gamma(2))$, such that $\varphi$ and $\varphi_0$ vanish at infinity (the $\varphi$ that vanish at the \emph{cusps} $0$ and $\infty$ of $\Gamma(2)$). 

By linear independence of the spaces $M_k(\Gamma(2))$, it suffices to define $\varphi^{\ast}:\mathcal{V}_d \rightarrow \C$  for $\varphi \in M_{d/2+m}(\Gamma(2))$, in which case the definition is
\begin{align*}
\varphi^{\ast}( (f_n), (g_n)) &= \sum_{n=1}^{\infty}{ \widehat{\varphi}(n)\int_{S^{d-1}}{f_n(\zeta) P_m(\zeta/ \sqrt{n})  d \zeta}}
- i^m \sum_{n=1}^{\infty}{ \widehat{\varphi_0}(n) \int_{S^{d-1}}{g_n(\zeta) P_m(\zeta/ \sqrt{n})  d \zeta}},
\end{align*}
where the series converge absolutely since the Fourier coefficients $\widehat{\varphi}(n)$  of $\varphi$ and $\widehat{\varphi_0}(n)$ of $\varphi_0$ are polynomially bounded. It now suffices to prove the following statements for all $\varphi \in \mathcal{I}_d$.
\begin{enumerate}[(i)]
\item $\varphi^{\ast}(\Phi_d(f)) =0$ for all $f \in  \mathcal{S}(\R^d)$.
\item $\varphi^{\ast} =0$, if and only if $\varphi = 0$.
\end{enumerate}
By continuity, it suffices to verify (i) for all Schwartz functions of the form $f(x) = w(x) e^{\pi i \tau |x|^2}$, with $w \in \cup_{m \geq 0}{\mathcal{B}_m}$ and $\tau \in \H$, since those functions generate a dense subspace of $\mathcal{S}(\R^d)$ (compare with the proof of part (iii) of Proposition \ref{prop:lifts}). In this case, the desired identity reduces  to the trivial identity $\varphi(\tau)- (-i \tau)^{-d/2-m} \varphi_0(\tau) = 0$, by orthogonality of spherical harmonics.  To prove assertion (ii), suppose that $\varphi^{\ast} = 0$, where $\varphi = \sum_{j=1}^{N}{\varphi_j} \in \mathcal{J}_d$,  $\varphi_j \in M_{d/2+m_j}(\Gamma(2))$ and $ m_1 < m_2 < \cdots < m_N$. Fix $n_0 \in \N$ and define $f_n, g_n \in C^{\infty}(S^{d-1})$ by 
\[
f_n(\zeta) = \delta_{n,n_0} \sum_{j=1}^N{\frac{1}{|\mathcal{B}_{m_j}|}\overline{P_{m_j}( \sqrt{n}\zeta)}},  \qquad g_n(\zeta) = 0.
\]
A short computation then shows that $\varphi^{\ast}((f_n),(g_n)) = \sum_{j=1}^N{\widehat{\varphi_j}(n_0)}=0$ and hence $\varphi = 0$, since $n_0$ was arbitrary.
\end{proof}
\section{More on the functions $b_{p,n}(r)$}\label{sec:bessel-kloosterman}
Here we present further connections of the series $F_p(\tau,r)$ defined in  \eqref{eq:def:Fp}, to classical Poincar{\'e} series by expressing the coefficients $b_{p,n}(r)$ defined in \eqref{eq:def-b-p-n} as a sum of Bessel functions times Kloosterman-type sums. By combining the formulas thus obtained with known estimates for Fourier coefficients of cusp forms, we will prove in Proposition \ref{prop:lower-bound-bpn} below  that infinitely many of the functions $b_{p,n}(r)$ are not of rapid decay. This points out two differences between the interpolation formula in Theorem \ref{thm:main-thm-radial-poincare} and the interpolation theorems in \cite{R-V,CKMRV-E8-leech}. The basis functions in those theorems are Schwartz functions, whose values at the interpolation nodes give, together with the values of their Fourier transforms, the ``natural" basis in a suitable space of pairs (or quadruples) of sequences of complex numbers. By contrast, we will prove that, whenever the $n$th Poincar{\'e} series of weight $p/2$ on $\Gamma_0(4)$ does not vanish identically, the function $b_{p,n}(r)$ does not have either of these properties.  For simplicity, we focus on even integers $p$ here and the functions $b_{p,n}(r)$. Similar results should hold for odd $p$ and the functions $\tilde{b}_{p,n}(r)$.

To start, we recall from  \S \ref{sec:a-particular-set-of-words} the definition of the set $\mathcal{B} \subset \overline{\Gamma(2)}$ and its basic properties. We choose  a complete set  of representatives $\mathcal{R}(\mathcal{B}) \subset \mathcal{B}$  for the quotient $\mathcal{B}/ \langle A \rangle$. By absolute and uniform convergence of the generating series $F_p(\tau,r)$, defined in \eqref{eq:def:Fp} and by \eqref{eq:def-b-p-n}, we have
\begin{align}
b_{p,n}(r) &= -\frac{1}{2} \sum_{M \in \mathcal{R}(\mathcal{B})}{\sum_{\ell \in \Z} \int_{iy_{0}+[-1,1]}{ (e^{\pi i \tau r^2}|_{p/2}(MA^{\ell})) e^{- \pi i n \tau}d \tau}} \notag \\
 &= -\frac{1}{2} \sum_{M \in \mathcal{R}(\mathcal{B})}{ \int_{iy_{0}+\R}{ (e^{\pi i \tau r^2}|_{p/2} M) e^{- \pi i n \tau}d \tau}}. \label{eq:pre-explicit-expansion-bpn}
\end{align}
The justification of the second equal sign is implied by assertion (i) of the Lemma \ref{lem:bessel-integral} below, which will be used to evaluate the above integrals. Before we give its statement, let us recall that the Bessel function $J_{\alpha}$ is given by
\[
J_{\alpha}(x) =  \left(\frac{x}{2}\right)^{\alpha}\,\sum_{j=0}^{\infty}{\frac{(-1)^j}{\Gamma(\alpha+1 +j)j!}\left(\frac{x}{2}\right)^{2j}}, \quad x , \alpha > 0.
\]
For integers $a, q$ with $q \geq 1$,  we define the Gauss sum $G_q(a) = \sum_{m =1}^{q}{e^{2\pi i a m^2/q }}$ and for any  coprime integers $c, d$ with $c > 0$, we define 
\[
g_c(d) =  \begin{cases} \frac{1}{2}G_{2c}(d) &\text{ if } c \equiv 0 \pmod{2},\\
 G_{c}(2d) & \text{ if }  c \equiv 1 \pmod{2}.
 \end{cases}
\]
Using Poisson summation one can verify that for any $ M = \begin{psmallmatrix}
 \ast & \ast \\  c & d
\end{psmallmatrix} \in \Gamma_{\theta}$ with $c > 0$, one has
\begin{equation}\label{eq:transformation-law-theta-3-with-gauss-sum}
\Theta_3(Mz) =  g_c(d) (- i(z + d/c))^{1/2} \Theta_3(z),
\end{equation}
for all $z \in \H$ (note that $z + d/c \in \H$, so \S \ref{sec:notation-square-roots} applies); see also \cite[pp. 28-33]{MumfordTata} for a detailed treatment on the transformation laws of $\Theta_3(\tau)$ and $\vartheta(z, \tau)$, or \cite[Thm 10.10]{Iwaniec} for the closely related function $\theta(z) = \Theta_3(2z)$. Raising \eqref{eq:transformation-law-theta-3-with-gauss-sum} to the eighth power, we deduce from \eqref{eq:automorphic-factors-the-same} that $g_c(d)^8 = c^4$ and in particular, $|g_c(d)| = \sqrt{c}$. 
\begin{lemma}\label{lem:bessel-integral}
For matrices $M = \begin{psmallmatrix}
a & b\\ c & d
\end{psmallmatrix} \in \Gamma_{\theta}
$ with $ c> 0$, real numbers $y_0 > 0$, $r \geq 0$ and integers $n$, $p$ such that $p \geq 5$, define the integral
\[
I_p(M, r, n, y_0) =  \int_{i y_{0}+\R}{ (e^{\pi i \tau r^2}|_{p/2} M) e^{- \pi i n \tau}d \tau}.
\]
\begin{enumerate}[(i)]
\item The integral $I_p(M, r, n, y_0)$ converges absolutely and is independent of $y_0$.
\item For all $n \leq 0$, we have $I_p(M, r, n, y_0) = 0$.
\item For all $n \geq 1$, we have $
I_p(M,0 ,n,y_0) = \frac{2 \pi (\pi n)^{p/2-1}}{\Gamma(p/2)}  g_c(d)^{-p} e^{\pi i \frac{d}{c}n}.
$
\item For all $n \geq 1$ and $r > 0$, we have \[
I_p(M,r ,n,y_0) =
 (2\pi) (n/r^2)^{p/4-1/2} c^{p/2-1} g_{c}(d)^{-p} e^{\pi i \frac{a}{c}r^2} e^{\pi i \frac{d}{c}n} J_{p/2-1}(2 \pi r \sqrt{n}/c).
\]
\end{enumerate}
\end{lemma}
The proof of Lemma \ref{lem:bessel-integral} closely follows standard computations in text books, for example \cite[Ch. 3.2]{Iwaniec}. We include them for completeness, convenience of the reader and because of the minor issue that the parameter $r^2$ is not an integer in our setting. 
\begin{proof}[Proof of Lemma \ref{lem:bessel-integral}]
We abbreviate by $g$ the function $g(\tau) =(e^{\pi i \tau r^2}|_{p/2} M) e^{- \pi i n \tau}$. 

For part (i), note that for $\tau = t + iy_0$, 
\begin{equation}\label{eq:estimate-for-bessel-integral}
 |g(\tau)| = |g(t + iy_0)| \leq \frac{e^{\pi n y_0}}{|c\tau + d|^{p/2}}  = \frac{e^{\pi n y_0}}{((ct +d)^2 + c^2 y_0^2)^{p/4}},
\end{equation}
which is an integrable function of $t$. Independence $y_0$ follows by applying Cauchy's Theorem to the function $g(\tau)$ and rectangles $y_0 \leq \imag(\tau) \leq y_1$, $|\real(\tau)| \leq R$, where $R \rightarrow \infty$. Alternatively, it follows from the formulas (iii) and (iv), to be proven below. 

Since we do not need part (ii) further below, we omit the simple proof, but we note that the statement of part (ii) would reprove Corollary \ref{cor:b_p-zero-for-negative-n}.

To prepare for parts (iii) and (iv), we write $M \tau = \frac{a}{c}- \frac{1}{c^2(\tau +d/c)}$ and use \eqref{eq:transformation-law-theta-3-with-gauss-sum} to write
\[
g(\tau) = (e^{\pi i \tau r^2}|_{p/2} M) e^{- \pi i n \tau} = \frac{e^{\pi i \frac{a}{c}r^2} e^{- \pi i \frac{r^2}{c^2(\tau+d/c)}}e^{- \pi i n (\tau + d/c -d/c)}}{g_c(d)^{p} (-i(\tau + d/c))^{p/2} }.
\]
By changing variables $\tau \leftrightarrow \tau + d/c$, we obtain
\[
I_p(M,r ,n,y_0) = g_c(d)^{-p} e^{\pi i \frac{a}{c}r^2} e^{\pi i \frac{d}{c}n} \mathcal{J}_p(r,c), \quad \text{where} \quad \mathcal{J}_p(r,c) = \int_{iy_0 + \R}{\frac{e^{ -\pi i \frac{r^2}{c^2 \tau}} e^{- \pi i n \tau}}{(-i \tau)^{p/2}}d\tau}.
\]
For the proof of part (iii), we need the formula $\int_{\R}{\frac{e^{i \nu t}}{(\eta+it)^z}dt} = \frac{(2\pi) e^{- \nu \eta} \nu^{z-1}}{\Gamma(z)}$, taken from \cite[8.315]{GR} and valid for $\real(z) > 1$, $\eta, \nu > 0$, where the argument of $\eta + it$ taken in $(-\pi/2, \pi /2)$, consistent with our convention form \S \ref{sec:notation-square-roots}. By writing $\tau = iy + t$, changing $t$ to $-t$ in the integral and applying the previous formula with $\eta = y_0$, $\nu = \pi n$, $z = p/2$, we obtain 
\[
\mathcal{J}_p(0, c) = e^{\pi n y_0}\int_{\R}{\frac{e^{\pi i n t}}{(y_0 + it)^{p/2}}} = e^{\pi n y_0} \frac{(2\pi) e^{-\pi n y_0} (\pi n)^{p/2-1}}{\Gamma(p/2)} = \frac{(2\pi) (\pi n)^{p/2-1}}{\Gamma(p/2)}.
\]
For the proof of part (iv), we introduce the variable $\beta = r^2/ c^2 > 0$. We write $e^{-\pi i (\beta/\tau)} = \sum_{j=0}^{\infty}{\tfrac{1}{j!}(- \pi i \beta/\tau)^j}$ and reduce to the case $r = 0$ considered (iii) in the following way: 
\begin{align*}
\mathcal{J}_p(r, c) &= \sum_{j=0}^{\infty}{ \frac{1}{j!} \left(- \pi i \beta \right)^{j}\int_{iy_0 + \R}{  \frac{ e^{- \pi i  n \tau} }{\tau^{j}(-i \tau)^{p/2}}   d\tau}} = \sum_{j=0}^{\infty}{ \frac{1}{j!} \left(- \pi \beta \right)^{j} \mathcal{J}_{p+2j}(0,c)}  \\
&= \sum_{j=0}^{\infty}{ \frac{(-1)^{j}}{j!} \left(\pi \beta \right)^{j} \frac{(2\pi) (\pi n)^{p/2+j-1}}{\Gamma(p/2+j)}} =  (2\pi) (n/\beta)^{p/4-1/2} J_{p/2-1}(2 \pi \sqrt{\beta n}) \qedhere
\end{align*}

%
%
\end{proof}
To proceed with the computation \eqref{eq:pre-explicit-expansion-bpn}, let us define, for all $n,c \in \N$ such that $c$ is even and all $r \in \C$,  the sum
\begin{equation}\label{eq:Kloosterman-type-sum}
S_p(r,n,c) =\sum_{d=1, \gcd(c,d) = 1}^{2c}{ (\sqrt{c}/ g_c(d))^{p} e^{\pi i ( \frac{\alpha(c,d)}{c}r^2 + \frac{dn}{c})} },
\end{equation}
where $ \alpha(c,d) \in \Z$  is defined by requiring that $\left[\begin{psmallmatrix}
 \alpha(c,d) & \ast\\ c & d
\end{psmallmatrix} \right] \in \mathcal{B}$, which is possible by Lemma \ref{lem:identifying-sets-B-tilde-B}. We can define an analogous sum $\tilde{S}_p(r,n,c)$ for all odd positive integers $c$. Inserting the formulas from Lemma \ref{lem:bessel-integral} into \eqref{eq:pre-explicit-expansion-bpn}, we obtain
\begin{align}
b_{p,n}(r) &= - \pi (n/r^2)^{p/4-1/2} \sum_{\substack{c=1\\ c \equiv 0 (2)}}^{\infty}{  \frac{1}{c} S_p(r,n,c) J_{p/2-1}(2\pi r \sqrt{n}/c) }, \quad r> 0, \label{eq:bpn-explicit-at-non-zero} \\
b_{p,n}(0) &= - \frac{\pi (\pi n)^{p/2-1}}{\Gamma(p/2)} \sum_{\substack{c=1\\ c \equiv 0 (2)}}^{\infty}{ \frac{1}{c^{p/2}} S_p(0,n,c) }.  \label{eq:bpn-explicit-at-zero}
\end{align}
A similar formula holds for $\tilde{b}_{p,n}(r)$ involving $\tilde{S}_p(r,n,c)$, where we sum over odd positive integers. 

Let us now specialize \eqref{eq:bpn-explicit-at-non-zero} to radii $r = \sqrt{m}$ with $m \in \N$ and to \emph{even} dimensions $p \geq 6$ and moreover introduce the notation $k= p/2 \in \Z_{\geq 3}$. We shall relate the values $b_{2k,n}(\sqrt{m})$ to Fourier coefficients of (actual) Poincar{\'e} series of weight $k$ on $\Gamma_0(4)$. To that end, we start by replacing $c$ by $c/2$ in \eqref{eq:bpn-explicit-at-non-zero} and correspondingly sum over $c \in  4\N$, giving
\begin{equation}\label{eq:massaging-b_pn-at-sqrt-m-1}
b_{2k,n}(\sqrt{m}) = - \pi \left( \frac{n}{m} \right)^{\frac{k-1}{2}} \sum_{\substack{c=1\\ c \equiv 0 (4)}}^{\infty}{  \frac{2}{c} S_p(\sqrt{m},n,c/2) J_{k-1}(4\pi  \sqrt{nm}/c) }.
\end{equation} 
Next, we rewrite the factor $\sqrt{c/2}/g_{c/2}(d)$ appearing in $S_p(\sqrt{m},n,c/2)$. For this, we use that for $4|c$ and $d$ coprime to $c$, we have $G_c(d)^2 = i (2c) \chi(d)$ , where $\chi: (\Z/4\Z)^{\times} \rightarrow \{-1, 1 \}$ denotes the non-trivial character. (This can be deduced from \eqref{eq:transformation-law-theta-3-with-gauss-sum} and \cite[eq. 2.73, p. 46]{Iwaniec}, for example.) Hence $( \sqrt{c/2}/g_{c/2}(d) )^{2k} =i^{-k} \chi(d)^k$.
%
%
Moreover, we have $\alpha(c/2,d) d \equiv 1 \pmod{c}$ and thus we can rewrite \eqref{eq:massaging-b_pn-at-sqrt-m-1} as
\begin{equation}\label{eq:b_kn-at-sqrt-m}
b_{2k,n}(\sqrt{m}) = - 2\pi i^{-k} \left( \frac{n}{m} \right)^{\frac{k-1}{2}} \sigma_k(m,n),
\end{equation}
where, abbreviating $e(w) = e^{2\pi i w}$ and writing $\bar{d}$ for the inverse of $d$ mod $c$, 
\begin{align*}
\sigma_k(m,n) &= \sum_{\substack{c=1\\ c \equiv 0 (4)}}^{\infty}{  \frac{1}{c} S_{\chi^k}(m,n,c) J_{k-1}(4\pi  \sqrt{nm}/c) },\\
S_{\chi^k}(m,n,c) &= \sum_{d \in (\Z/c\Z)^{\times}}{  \chi(d)^k e\left( \frac{ \bar{d}m + nd}{c}\right)}.
\end{align*}
Consider the $m$th Poincar{\'e} series for $\Gamma_0(4)$ of weight $k$ and character $\chi^k \in \{\chi, 1 \}$:
\begin{equation}\label{eq:definition-actual-poincare-series}
P_m(z) = \sum_{\gamma \in \Gamma_{\infty} \backslash \Gamma_0(4)}{\chi(\gamma)^k (c_{\gamma}z + d_{\gamma})^{-k}e(m(\gamma z))},
\end{equation}
where $\Gamma_{\infty} \subset \Gamma_0(4)$ denotes the subgroup of upper triangular matrices, where $c_{\gamma}, d_{\gamma}$ denote the bottom row entries of $\gamma$ and where we suppress the dependence on $k$ and hence $\chi$ from the notation. It is well-known \cite[Lemma 14.2]{Kowalski-Iwaniec} that  its $n$th Fourier coefficient is given by
\begin{equation}\label{eq:nth-FC-of-P_m}
\widehat{P_m}(n) = 2 \pi i^{-k} \left( \frac{n}{m} \right)^{\frac{k-1}{2}} \left( \delta(m,n) + \sigma_k(m,n) \right).
\end{equation}
Comparing \eqref{eq:nth-FC-of-P_m} with \eqref{eq:b_kn-at-sqrt-m}  we deduce that $b_{2k,n}(\sqrt{m}) = - \widehat{P_m}(n)$, for all $m,n \in \N$ such that $n \neq m$. The following proposition shows that, for infinitely many $n$, the upper bounds for $b_{p,n}(r)$ given in \eqref{eq:bounds-for-bpn-nonzero-r-in-thm} in Theorem \ref{thm:main-thm-radial-poincare}, can't be significantly improved: the term $r^{-(p/2) +2 + \varepsilon}$ cannot be replaced by $r^{-(p/2) + 1- \varepsilon}$.
\begin{proposition}\label{prop:lower-bound-bpn}
Fix an even integer $p \geq 6$ and let $k = p/2$. There exist infinitely many integers $n \geq 1$ with the following property. For each $\varepsilon > 0$, the function $r \mapsto r^{k-1+\varepsilon}b_{2k,n}(r)$ is unbounded on $(0, +\infty)$, in fact, unbounded on the subset of $r = \sqrt{m}$, $m \in \N$. 
\end{proposition}
\begin{proof}
We first recall that, for any $n \geq 1$, taking the Petersson inner product of the $n$th Poincar{\'e} series $P_n$ (as defined in \eqref{eq:definition-actual-poincare-series}) with any $f \in M_k(\Gamma_0(4), \chi^k)$ returns the $n$th Fourier coefficient $\widehat{f}(n)$ of $f$, up to nonzero scalars (see \cite[Lemma 14.3]{Kowalski-Iwaniec}). Since the space $M_k(\Gamma_0(4), \chi^k)$ is nonzero and a nonzero modular form has infinitely many nonzero Fourier coefficients, there are infinitely many indices $n$, for which $P_n$ does not vanish identically.

Fix an index $n \in \N$ such that $P_n \neq 0$ and assume that for some $A > 0$, we have $b_{2k,n}(r) = O(r^{-A})$, as $r \rightarrow \infty$. We will show that $A \leq k-1$. Our main calculation $b_{2k,n}(\sqrt{m}) = - \widehat{P_m}(n)$ and \eqref{eq:b_kn-at-sqrt-m}, \eqref{eq:nth-FC-of-P_m} imply that, for all $m \in \N \setminus \{n\}$,
\[
|b_{2k,n}(\sqrt{m})| = |\widehat{P_m}(n)| = \left( \frac{n}{m} \right)^{k-1} |\widehat{P_n}(m)|.
\]
Our assumption then gives $|\widehat{P_n}(m)| = O(m^{-A/2+k-1})$, as $m \rightarrow \infty$. In particular, we have $M^{-k}\sum_{m=1}^{M}{|\widehat{f}(m)|^2} = O(M^{-A+k-1})$, as $M \rightarrow \infty$. On the other hand,  the Rankin--Selberg method (see \cite[Theorem 1 and Remark B on page 364]{rankin-rankin-selberg}) implies that for every $f\in S_k(\Gamma_0(4),\chi^k)$, the Fourier coefficients $\widehat{f}(m)$  satisfy
\[
\sum_{m=1}^{M}{|\widehat{f}(m)|^2} = c_k(f) M^{k} + O(M^{k-2/5}), \qquad M \rightarrow \infty,
\]
where $c_k(f)$ is proportional to the Peterson norm of $f$ (which is $>0$, if and only $f \neq 0$). Taking $f =P_n$ here and comparing the two asymptotic relations, we deduce $A \leq k-1$.
\end{proof}

\begin{remark-non}
Little is known about the (non-)vanishing of individual Poincar{\'e} seires, even in level $1$. A result of Mozzochi \cite{mozzochi}[Thm 2], that builds up on the work of Rankin \cite{rankin-poincare} in level 1, implies that, for all sufficiently large even integral weights $k$, the first $O(k^{1.99})$ Poincar{\'e} series of weight $k$ on $\Gamma_0(N)$, do not vanish identically.

\end{remark-non}

\subsection*{Acknowledgments}
I would like to thank Maryna Viazovska, my doctoral advisor, for her valuable ideas, her support and guidance over the course of this work. That a formula like the one in Theorem \ref{thm:main-thm-non-radial} could exist, was conjectured by Danylo Radchenko, whom I thank for several enlightening discussions and suggestions. I also thank Matthew de Courcy-Ireland for giving me detailed feedback on earlier drafts of this paper and helpful discussions.


\begin{thebibliography}{10}
\bibitem{AxlerBourdonRamey}
Sheldon Axler, Paul Bourdon, and Wade Ramey.
\newblock {\em Harmonic Function Theory}.
\newblock {Springer}, 2001.

\bibitem{bochner-theta-relations}
S.~Bochner.
\newblock {\em Theta relations with spherical harmonics.}
\newblock {Proceedings of the National Academy of Sciences of the United States of America}, 37(12):804--808, 1951.
  
\bibitem{BRS}
Andiry Bondarenko, Danylo Radchenko, Kristian Seip.
\newblock{Fourier interpolation with zeros of zeta and L-functions.}
\newblock { arXiv e-prints,  \href{https://arxiv.org/abs/2005.02996}{2005.02996}, 2020}.

\bibitem{CKMRV-24}
Henry Cohn, Abhinav Kumar, Stephen~D. Miller, Danylo Radchenko, and Maryna
  Viazovska.
\newblock {\em The sphere packing problem in dimension 24.}
\newblock { Annals of Mathematics}, 185(3):1017--1033, 2017.

\bibitem{CohnConcalves}
Henry Cohn and Felipe Gon{\c{c}}alves.
\newblock {\em An optimal uncertainty principle in twelve dimensions via modular
  forms.}
\newblock {Inventiones mathematicae}, 217(3):799--831, 2019.


  
\bibitem{CKMRV-E8-leech}
Henry {Cohn}, Abhinav {Kumar}, Stephen~D. {Miller}, Danylo {Radchenko}, and
  Maryna {Viazovska}.
\newblock {\em Universal optimality of the $E_8$ and Leech lattices and
  interpolation formulas.}
  \newblock { Annals of Mathematics}, to appear.

\bibitem{GR}
Gradshteyn, I. S. and Ryzhik, I. M.
\newblock{\em Table of integrals, series, and products.} 
\newblock{Elsevier/Academic Press, Amsterdam}, 2007.


\bibitem{Grafakos}
Loukas Grafakos and Gerald Teschl.
\newblock {\em On fourier transforms of radial functions and distributions.}
\newblock { Journal of Fourier Analysis and Applications}, 19(1):167--179,  2013.

\bibitem{Howe-Tan}
Roger~E. Howe and Eng~C. Tan.
\newblock {\em Non-Abelian Harmonic Analysis: Applications of $\SL(2,\R)$.}
\newblock Universitext. Springer New York, 2012.
\bibitem{Iwaniec}
Henryk Iwaniec.
\newblock {\em Topics in Classical Automorphic Forms.}
\newblock American Mathematical Society, 1997.



\bibitem{Knopp-eichler-coh}
Marvin~I. Knopp.
\newblock {\em Some new results on the Eichler cohomology of automorphic forms.}
\newblock { Bulletin of the American Mathematical Society}, 80(4):607--632, 07, 1974.
\bibitem{Kowalski-Iwaniec}
Emmanuel Kowalski, Henryk Iwaniec.
\newblock{Analytic Number Theory}, American Mathematical Society, 2004.
\bibitem{mozzochi}
C Mozzochi.
\newblock{On the non-vanishing of Poincaré series}.
\newblock{Proceedings of the Edinburgh Mathematical Society}, 32(1), 131-137, 1989.
\bibitem{MumfordTata}
David Mumford.
\newblock {\em Tata lectures on Theta, I}.
\newblock Birkh{\"a}user, 1994.
\bibitem{R-V}
Danylo Radchenko and Maryna Viazovska.
\newblock {\em Fourier interpolation on the real line.}
\newblock { Publications math{\'e}matiques de l'IH{\'E}S}, 129(1):51--81, 2019.
\bibitem{Ramos-Sousa}
Jo{\~a}o P.~G. {Ramos} and Mateus {Sousa}.
\newblock {\em Fourier uniqueness pairs of powers of integers}.
\newblock {Journal of the European Mathematical Society}, to appear.


\bibitem{rankin-rankin-selberg}
Robert Rankin. 
\newblock{Contributions to the theory of Ramanujan's function $\tau(n)$ and similar arithmetical functions: II. The order of the Fourier coefficients of integral modular forms.}
\newblock{Mathematical Proceedings of the Cambridge Philosophical Society}, 35(3), 357-372, 1939.
 
\bibitem{rankin-poincare}
Robert Rankin. 
\newblock{The vanishing of Poincaré series}.
\newblock{Proceedings of the Edinburgh Mathematical Society}. 23(2), 151-161. 1980.

\bibitem{rankin-book}
Robert Rankin.
\newblock {\em Modular Forms and Functions}.
\newblock {Cambridge University Press}, 1977.





\bibitem{SerreCourse}
Jean-Pierre Serre.
\newblock {\em A Course in Arithmetic}.
\newblock {Springer}, 1973.
\bibitem{Stein-Weiss}
Elias~M. Stein and Guido Weiss.
\newblock{\em Introduction to Fourier Analysis on Euclidean Spaces}.
\newblock{Princeton University Press}, 1971.
\bibitem{Viazovska8}
Maryna~S. Viazovska.
\newblock {\em The sphere packing problem in dimension 8.}
\newblock {Annals of Mathematics}, 185(3):991--1015, 2017.
\bibitem{Whitney}
Hassler Whitney.
\newblock{\em Differentiable even functions.}
\newblock{Duke Mathematical Journal}, 10(1):159--160, 03, 1943.

\end{thebibliography}

\end{document}